\documentclass[11pt]{article}
\usepackage{appendix,latexsym,amsfonts,amsmath,amssymb,graphicx,hyperref,amsthm,soul,verbatim,authblk}
\usepackage[framemethod=tikz]{mdframed}

\newcommand{\EE}{\mbox{\bf E}\,}
\newcommand{\PP}{\mathbb{\bf P}\,}

\newcommand{\R}{\mathbb{R}}
\newcommand{\C}{\mathbb{C}}

\newcommand{\HH}{\mathbb{H}}
\newcommand{\N}{\mathbb{N}}
\newcommand{\D}{\mathbb{D}}
\newcommand{\TT}{\mathbb{T}}

\newcommand{\Ree}{\mbox{Re}\,}
\newcommand{\Imm}{\mbox{Im}\,}

\newcommand{\pa}{\partial}

\newcommand{\no}{\noindent}

\newcommand{\BGE}{\begin{equation}}
\newcommand{\BGEN}{\begin{equation*}}
\newcommand{\EDE}{\end{equation}}
\newcommand{\EDEN}{\end{equation*}}

\def\til{\widetilde}
\def\ha{\widehat}
\def\sem{\setminus}
\def\lin{\overline}

\def\conf{\stackrel{\rm Conf}{\twoheadrightarrow}}
\def\luto{\stackrel{\rm l.u.}{\longrightarrow}}
\def\dto{\stackrel{\rm Cara}{\longrightarrow}}

\DeclareMathOperator{\ccap}{cap} 
 \DeclareMathOperator{\diam}{diam}
\DeclareMathOperator{\dist}{dist} \DeclareMathOperator{\dcap}{dcap}
\DeclareMathOperator{\hcap}{hcap} \DeclareMathOperator{\id}{id}

 \DeclareMathOperator{\SLE}{SLE}

\newtheorem{Theorem}{Theorem}[section]
\newtheorem{Lemma}[Theorem]{Lemma}
\newtheorem{Definition}[Theorem]{Definition}
\newtheorem{Corollary}[Theorem]{Corollary}
\newtheorem{Proposition}[Theorem]{Proposition}

\numberwithin{equation}{section}

\begin{document}
\title{Ergodicity of the tip of an SLE curve}
\author{Dapeng Zhan\footnote{Research partially supported by NSF grants DMS-1056840 and Sloan fellowship}}
\affil{Michigan State University}
\maketitle
\abstract{We first prove that, for $\kappa\in(0,4)$, a whole-plane SLE$(\kappa;\kappa+2)$ trace stopped at a fixed capacity time satisfies reversibility. We then use this reversibility result to prove that, for $\kappa\in(0,4)$, a chordal SLE$_\kappa$ curve stopped at a fixed capacity time can be mapped conformally to the initial segment of a whole-plane SLE$(\kappa;\kappa+2)$ trace. A similar but weaker result holds for radial SLE$_\kappa$. These results are then used to study the ergodic behavior of an SLE curve near its tip point at a fixed capacity time. The proofs rely on the symmetry of backward SLE weldings and conformal removability of SLE$_\kappa$ curves for $\kappa\in(0,4)$.}

\section{Introduction}
The Schramm-Loewner Evolution $\SLE_{\kappa}$, introduced by Oded Schramm, generates random curves in plane domains, which are the scaling limits of a number of critical two dimensional lattice models. Many work have been done to prove the convergence of various discrete models to SLE with different parameters $\kappa$. It is also interesting to study the geometric properties of the SLE curves.

The current paper focuses on studying the tips of two versions of SLE: chordal SLE and radial SLE at some fixed capacity time.
There were previous work on the tips of SLE, e.g., \cite{tip-fractal}, in which the multifractal spectrum of the SLE tip is studied.
This paper studies the ergodic property of the SLE near its tip. Consider an SLE$_\kappa$ ($\kappa\in(0,4)$) curve $\beta$, which is a simple curve. Let $h:[0,1]\to [-\infty,T]$ be a decreasing function such that $h(1)=-\infty$ and the (logarithm) capacity of $\beta([t,1])$ is $h(t)$. Then we claim that the harmonic measure of a particular side of $\beta([h^{-1}(t),1])$ changes in an ergodic way as $t\to -\infty$. See Section \ref{ergodicity}.

We will use results about backward SLE derived in \cite{RZ}. The traditional chordal or radial SLE$_\kappa$ is defined by solving a chordal or radial Loewner equation driven by $\sqrt\kappa B(t)$. Adding a minus sign to the (forward) Loewner equations, we get the backward Loewner equations. The backward chordal or radial SLE$_\kappa$ is then defined by solving a backward chordal or radial Loewner equation driven by $\sqrt\kappa B(t)$.


The backward radial SLE$(\kappa;\rho)$ processes resemble the forward radial SLE$(\kappa;\rho)$ processes, and play an important role in this paper. If $\kappa\in(0,4]$ and $\rho\le -\frac\kappa 2-2$, a backward radial SLE$(\kappa;\rho)$ process induces a random welding $\phi$, which is an involution (an auto homeomorphism whose inverse is itself) of the unit disc with exactly two fixed points such that for $w\ne z$, $w=\phi(z)$ iff $f_t(z)=f_t(w)$ when $t$ is big enough, where $(f_t)$ are the solutions of the backward Loewner equation. It is proven in \cite{RZ} that, for $\kappa\in(0,4]$, there is a coupling of two different backward radial SLE$(\kappa;-\kappa-6)$ processes, which induce the same welding.

In Section \ref{normalized} of this paper, we use a limit procedure to define a normalized backward radial SLE$(\kappa;\rho)$ trace, and prove that, up to a reflection about the unit circle, it agrees with the forward whole-plane SLE$(\kappa;-4-\rho)$ curve (Theorem \ref{cor-whole}). Using the symmetry of backward radial SLE$(\kappa;-\kappa-6)$ welding together with the conformal removability of SLE$_\kappa$ curves, we prove in Section \ref{image} that, for $\kappa\in(0,4)$, a whole-plane SLE$(\kappa;\kappa+2)$ curve stopped at time $0$ satisfies reversibility (Theorem \ref{reversibility-whole}). This result is then used to prove that, for $\kappa\in(0,4)$, a forward chordal SLE$_\kappa$ curve stopped at a fixed capacity time can be mapped conformally to an initial segment of a whole-plane SLE$(\kappa;\kappa+2)$ curve, and the same is true up to a change of the probability measure for a forward radial SLE$_\kappa$ (Theorems \ref{Thm2} and \ref{Thm-R}). In Section \ref{ergodicity}, we use the above conformal relations to derive ergodic properties of a chordal or radial SLE$_\kappa$ curves at a fixed capacity time (Theorem \ref{chordal-radial-ergodicity}).

Throughout this paper, we use the following symbols and notation. Let $\ha\C=\C\cup\{\infty\}$, $\D=\{z\in\C:|z|<1\}$,  $\D^*=\ha\C\sem\lin\D$,
$\TT=\{z\in\C:|z|=1\}$, and $\HH=\{z\in\C:\Imm z>0\}$. Let $\cot_2(z)=\cot(z/2)$ and $\sin_2(z)=\sin(z/2)$. Let $I_\TT(z)=1/\lin z$ be the reflections about $\TT$, respectively. By an interval on $\TT$, we mean a connected subset of $\TT$. We use $B(t)$ to denote a standard real Brownian motion. We use $C(J)$ to denote the space of real valued continuous functions on $J$. By $f:D\conf E$ we mean that $f$ maps $D$ conformally onto $E$. By $f_n\luto f$ in $U$ we mean that $f_n$ converges to $f$ locally uniformly in $U$.
\vskip 4mm
\no{\bf Acknowledgements.} I  would like to thank Gregory Lawler for helpful discussions about radial Bessel processes.

\section{Loewner Equations} \label{Loewner}
\subsection{Forward equations}
We review the definitions and basic facts about (forward) Loewner equations. The reader is referred to \cite{LawSLE} for details.

A set $K$ is called an $\HH$-hull if it is a bounded relatively closed subset of $\HH$, and $\HH\sem K$ is simply connected. For every $\HH$-hull $K$, there is a unique $g_K:\HH\sem K\conf \HH$ such that $g_K(z)-z\to 0$ as $z\to \infty$. The number $\hcap(K):=\lim_{z\to\infty} z(g_K(z)-z)$ is always nonnegative, and is called the half plane capacity of $K$. A set $K$ is called a $\D$-hull if it is a relatively closed subset of $\D$, does not contain $0$, and $\D\sem K$ is simply connected. For every $\D$-hull $K$, there is a unique $g_K:\D\sem K\conf \D$ such that $g_K(0)=0$ and $g_K'(0)>0$. The number $\dcap(K):=\log(g_K'(0))$ is always nonnegative, and is called the disc capacity of $K$. A set $K$ is called a $\C$-hull if it is a connected compact subset of $\C$ such that $\C\sem K$ is connected. For every $\C$-hull with more than one point, $\ha\C\sem K$ is simply connected, and there is a unique $g_K:\ha\C\sem K\conf \D^*$ such that $g_K(\infty)=\infty$ and $g_K'(\infty):=\lim_{z\to\infty} z/g_K(z)>0$. The real number $\ccap(K):=\log(g_K'(\infty))$ is called the whole-plane capacity of $K$. In either of the three cases, let $f_K=g_K^{-1}$.

Let $\lambda\in C([0,T))$, where $T\in(0,\infty]$. The chordal Loewner equation driven by $\lambda$ is
$$\pa_t g_t(z)=\frac{2}{g_t(z)-\lambda(t)},\quad 0\le t<T;\quad \quad g_0(z)=z.$$ 
The radial Loewner equation driven by $\lambda$ is
$$\pa_t g_t(z)=g_t(z)\frac{e^{i\lambda(t)}+g_t(z)}{e^{i\lambda(t)}-g_t(z)},\quad 0\le t<T;\quad g_0(z)=z. $$ 
Let $g_t$, $0\le t<T$, be the solutions of the chordal (resp.\ radial) Loewner equation.
For each $t\in[0,T)$, let $K_t$ be the set of $z\in\HH$ (resp.\ $\in\D$) at which $g_t$ is not defined. Then for each $t$, $K_t$ is an $\HH$(resp.\ $\D$)-hull with $\hcap(K_t)=2t$ (resp.\ $\dcap(K_t)=t$) and $g_{K_t}=g_t$. We call $g_t$ and $K_t$, $0\le t<T$,  the chordal (resp.\ radial) Loewner maps and hulls driven by $\lambda$.
We say that the process generates a chordal (resp.\ radial) trace $\beta$ if each $g_t^{-1}$ extends continuously to $\lin\HH$ (resp.\ $\lin\D$), and $\beta(t):= g_t^{-1}(\lambda(t))$ (resp.\ $:=g_t^{-1}(e^{i\lambda(t)})$), $0\le t<T$, is a continuous curve in $\lin\HH$ (resp.\ $\lin\D$). If the chordal (resp.\ radial) trace $\beta$ exists, then for each $t$, $K_t$ is the $\HH$-hull generated by $\beta([0,t])$, i.e., $\HH\sem K_t$ (resp.\ $\D\sem K_t$) is the component of $\HH\sem \beta([0,t])$ (resp.\ $\D\sem\beta([0,t])$) which is unbounded (resp.\ contains $0$).  Note that $\beta(0)=\lambda(0)\in\R$ (resp.\ $=e^{i\lambda(0)}\in\TT$). The trace $\beta$ is called $\HH$-simple (resp.\ $\D$-simple) if it has no self-intersections and intersects $\R$ (resp.\ $\TT$) only at its one end point, in which case we have $K_t=\beta((0,t])$ for $0\le t<T$. Since $\hcap(K_t)=2t$ (resp.\ $\dcap(K_t)=t$) for all $t$, we say that the chordal (resp.\ radial) trace is parameterized by the half-plane (resp.\ disc) capacity.


A simple property of the chordal (resp.\ radial) Loewner process is the translation (resp.\ rotation) symmetry. Let $C\in\R$ and $\lambda^*=\lambda+C$. Let $g^*_t$ and $K^*_t$ be the chordal (resp.\ radial) Loewner maps and hulls driven by $\lambda^*$. Then $K^*_t=C+K_t$ and $g^*_t(z)=C+g_t(z-C)$ (resp.\ $K^*_t=e^{iC} K_t$ and $g^*_t(z)=e^{iC}g_t(z/e^{iC})$). If $\lambda$ generates a chordal (resp.\ radial) trace $\beta$, then $\lambda^*$ also generates a chordal (resp.\ radial) trace $\beta^*$ such that $\beta^*=C+\beta$ (resp.\ $=e^{iC}\beta$).

Let $\kappa>0$. The chordal (resp.\ radial) SLE$_\kappa$ is defined by solving the chordal (resp.\ radial) Loewner equation with $\lambda(t)=\sqrt\kappa B(t)$, and the process a.s.\ generates a chordal (resp.\ radial) trace, which is $\HH$(resp.\ $\D$)-simple if $\kappa\in(0,4]$.

Let $T\in\R$ and $\lambda\in C((-\infty,T])$. The whole-plane Loewner equation driven by $\lambda$ is
$$\left\{\begin{array}{ll} \pa_t g_t(z)=g_t(z)\frac{e^{i\lambda(t)}+g_t(z)}{e^{i\lambda(t)}-g_t(z)},&t\le T;\\
\lim_{t\to-\infty} e^t g_t(z)=z,&z\ne 0.
\end{array}\right.$$
It turns out that the family $(g_t)$ always exists, and is uniquely determined by $(e^{i\lambda(t)})$. Moreover, there is an increasing family of $\C$-hulls $(K_t)_{-\infty<t\le T}$ in $\C$ with $\bigcap_t K_t=\{0\}$ such that $\ccap(K_t)=t$ and $g_{K_t}=g_t$. We call $g_t$ and $K_t$, $-\infty<t<T$, the whole-plane Loewner maps and hulls driven by $\lambda$. We say that the process generates a whole-plane trace $\beta$ if each $g_t^{-1}$ extends continuously to $\lin{\D^*}$, and $\beta(t):= g_t^{-1}(e^{i\lambda(t)})$, $-\infty< t<T$, is a continuous curve in $\C$. If the whole-plane trace $\beta$ exists, then it extends continuously to $[-\infty,T]$ with $\beta(-\infty)=0$, and for every $t$, $\C\sem K_t$ is the unbounded component of $\C\sem\beta([-\infty,t])$. If $\beta$ is a simple curve, then $K_t=\beta([-\infty,t])$ for every $t$. So we say that the whole-plane trace is parameterized by the whole-plane capacity.


\subsection{Backward equations}
Now we review the definitions and basic facts about backward Loewner equations. The reader is referred to \cite{RZ} for details.

Let $T\in(0,\infty]$ and  $\lambda\in C([0,T))$. The backward chordal Loewner equation driven by $\lambda$ is
$$ \pa_t f_t(z)=-\frac{2}{f_t(z)-\lambda(t)},\quad 0\le t<T; \quad f_0(z)=z. $$ 
The backward radial Loewner process driven by $\lambda$ is
$$ \pa_t f_t(z)=-f_t(z)\frac{e^{i\lambda(t)}+f_t(z)}{e^{i\lambda(t)}-f_t(z)},\quad 0\le t<T;\quad f_0(z)=z. $$ 
Let $f_t$, $0\le t<T$, be the solutions of the backward chordal (resp.\ radial) Loewner equation.
Let $L_t=\HH\sem f_t(\HH)$ (resp.\ $\D\sem f_t(\D)$), $0\le t<T$. Then every $L_t$ is an $\HH$(resp.\ $\D$)-hull with $\hcap(L_t)=2t$ (resp.\ $\dcap(L_t)=t$) and $f_{L_t}=f_t$. We call $f_t$ and $L_t$, $0\le t<T$, the backward chordal (resp.\ radial) Loewner maps and hulls driven by $\lambda$.

Define a family of maps $f_{t_2,t_1}$, $t_1,t_2\in[0,T)$, such that, for any fixed $t_1\in[0,T)$ and $z\in\ha\C\sem\{\lambda(t_1)\}$, the function $t_2\mapsto f_{t_2,t_1}(z)$ is the solution of the first (resp.\ second) equation below (with the maximal definition interval):
$$ \pa_{t_2} f_{t_2,t_1}(z)=-\frac{2}{f_{t_2,t_1}(z)-\lambda(t_2)},\quad f_{t_1,t_1}(z)=z; $$ 
\BGE \pa_{t_2} f_{t_2,t_1}(z)=-f_{t_2,t_1}(z)\frac{e^{i\lambda(t_2)}+f_{t_2,t_1}(z)}{e^{i\lambda(t_2)}-f_{t_2,t_1}(z)},\quad f_{t_1,t_1}(z)=z.\label{radial-backward-2}\EDE
We call $(f_{t_2,t_1})$ the backward chordal (resp.\ radial) Loewner flow driven by $\lambda$. Note that we allow that $t_2$ to be smaller than $t_1$ if $t_1>0$. If $t_2\ge t_1$, $f_{t_2,t_1}$ is defined on the whole $\HH$ (resp.\ $\D$); and this is not the case if $t_2<t_1$.
The following lemma is obvious.

\begin{Lemma}
\begin{enumerate}
  \item[(i)] For any $t_1,t_2,t_3\in[0,T)$, $f_{t_3,t_2}\circ f_{t_2,t_1}$ is a restriction of $f_{t_3,t_1}$. In particular, this implies that $f_{t_1,t_2}=f_{t_2,t_1}^{-1}$.
  \item[(ii)] For any fixed $t_0\in[0,T)$, $f_{t_0+t,t_0}$, $0\le t<T-t_0$, are the backward chordal (resp.\ radial) Loewner maps driven by $\lambda(t_0+t)$, $0\le t<T-t_0$. Especially, $f_{t,0}=f_t$, $0\le t<T$.
  \item[(iii)] For any fixed $t_0\in[0,T)$, $f_{t_0-t,t_0}$, $0\le t\le t_0$, are the forward chordal (resp.\ radial) Loewner maps driven by $\lambda(t_0-t)$, $0\le t\le t_0$.
\end{enumerate} \label{ft2t1}
\end{Lemma}



We say that a backward chordal (resp.\ radial) Loewner process driven by $\lambda\in C([0,T))$ generates a family of backward chordal (resp.\ radial) traces $\beta_t$, $0\le t\le t_0$, if for each fixed $t_0\in(0,T)$, the forward chordal (resp.\ radial) Loewner process driven by $\lambda(t_0-t)$, $0\le t\le t_0$, generates a chordal (resp.\ radial) trace, which is $\beta_{t_0}(t_0-t)$, $0\le t\le t_0$. Equivalently, this means that, for each $t_0$, $\beta_{t_0}:[0,t_0]\to\lin{\HH}$ (resp.\ $\lin\D$) is continuous, and or any $t_2\ge t_1\ge 0$, $f_{t_2,t_1}$ extends continuously to $\lin\HH$ (resp.\ $\lin{\D}$) such that $\beta_{t_2}(t_1)=f_{t_2,t_1}({\lambda(t_1)})$ (resp.\ $f_{t_2,t_1}(e^{i\lambda(t_1)})$). Taking $t_2=t_1=t$, we get $\beta_t(t)=\lambda(t)\in\R$ (resp. $=e^{i\lambda(t)}\in\TT$). Moreover, the equality  $f_{t_2,t_1}\circ f_{t_1,t_0}=f_{t_2,t_0}$, $t_2\ge t_1\ge t_0\ge 0$, holds after the continuation, and so we have
\BGE f_{t_2,t_1}(\beta_{t_1}(t))=\beta_{t_2}(t),\quad t_2\ge t_1\ge t\ge 0.\label{backward-trace}\EDE

The backward chordal (resp.\ radial) SLE$_\kappa$ is defined to be the backward chordal (resp.\ radial) Loewner process driven by $\sqrt\kappa B(t)$, $0\le t<\infty$. The existence of the forward chordal (resp.\ radial SLE$_\kappa$) trace together with Lemma \ref{ft2t1} and the translation (resp.\ rotation) symmetry implies that the backward chordal (resp.\ radial) SLE$_\kappa$ process generates a family of backward chordal (resp.\ radial) traces, which are $\HH$(resp.\ $\D$)-simple, if $\kappa\le 4$. 

\vskip 4mm
\no{\bf Remark.} One should keep in mind that each $\beta_t$ is a continuous function defined on $[0,t]$, $\beta_t(0)$ is the tip of $\beta_t$, and $\beta_t(t)$ is the root of $\beta_t$, which lies on $\R$. The parametrization is different from a forward chordal Loewner trace.
\vskip 4mm

For every $\HH$(resp.\ $\D$)-hull $L$, $g_L$ extends analytically to $\R\sem\lin{L}$ (resp.\ $\TT\sem\lin{L}$), and maps $\R\sem\lin{L}$ (resp.\ $\TT\sem\lin{L}$) to an open subset of $\R$ (resp.\ $\TT$). The set $S_L:=\R\sem g_L(\R\sem\lin{L})$ (resp.\ $:=\TT\sem g_L(\TT\sem\lin{L})$) is a compact subset of $\R$ (resp.\ $\TT$), and is called the support of $L$. The map $f_L$ then extends analytically to $\R\sem S_L$ (resp.\ $\TT\sem S_L$). If $(L_t)_{0\le t<T}$ are $\HH$(resp.\ $\D$)-hulls generated by a backward chordal (resp.\ radial) Loewner process, then each $S_{L_t}$ is an interval on $\R$ (resp.\ $\TT$), and $S_{L_{t_1}}\subset S_{L_{t_2}}$ if $t_1<t_2$ (c.f.\ Lemmas 2.7 and 3.3 in \cite{RZ}). The following is Lemma 3.5 in \cite{RZ}.

\begin{Lemma}
  Let $L_t$, $0\le t<\infty$, be $\D$-hulls generated by a backward radial Loewner process. Then $\bigcup_t S_{L_t}$ is equal to either  $\TT$ or  $\TT$ without a single point. \label{union-support}
\end{Lemma}

Now we review the welding induced by a backward Loewner process. See Section 3.5 of \cite{RZ} for details.

Suppose $L=\beta$ is an $\HH$(resp.\ $\D$)-simple curve. Then $S_\beta$ is the union of two intervals on $\R$ (resp.\ $\TT$), which intersects at one point, and $f_\beta$ extends continuously to $S_\beta$, and maps the two intervals onto the two sides of $\beta$. Every point on $\beta$ except the tip point has two preimages. The welding $\phi_\beta$ induced by $\beta$ is the involution of $S_\beta$ with exactly one fixed point, which is the $f_\beta$-pre-image of the tip of $\beta$, such that for $x\ne y\in S_\beta$, $y=\phi_\beta(x)$ if and only if $f_{\beta}(x)=f_{\beta}(y)$.

Suppose a backward chordal (resp.\ radial) Loewner process generates a family of  $\HH$(resp.\ $\D$)-simple traces $(\beta_t)_{0\le t<T}$. Then for any $t_1<t_2$, $S_{\beta_{t_1}}$ is contained in the interior of $S_{\beta_{t_2}}$, and $\phi_{\beta_{t_1}}$ is a restriction of $\phi_{\beta_{t_2}}$. The latter can be seen from $f_{t_2,t_1}\circ f_{t_1}=f_{t_2}$. So the process naturally induces a welding $\phi$, which is an involution of the open interval $\bigcup_{0\le t<T} S_{\beta_t}$ on $\R$ (resp.\ $\TT$) such that $\phi|_{S_{\beta_t}}=\phi_{\beta_t}$ for each $t$. The lamination has only one fixed point, which is $\lambda(0)\in\R$ (resp.\ $e^{i\lambda(0)}\in\TT$). Consider the radial case and suppose $T=\infty$. Lemma \ref{union-support} and the properties of  $S_{\beta_t}$'s imply that $\TT\sem \bigcup_{0\le t<\infty} S_{\beta_t}$ contains exactly one point, say $w_0$. We call $w_0$ the joint point of the process, which is the only point such that $f_t(w_0)\in\TT$ for all $t\ge 0$.  In this case we extend $\phi$ to an involution of $\TT$ with exactly two fixed points: $e^{i\lambda(0)}$ and $w_0$.

\section{SLE$(\kappa;\rho)$ Processes} \label{Section-kappa-rho'}
In this section, we review the definitions of the forward and backward radial SLE$(\kappa;\rho)$ processes, respectively, as well as the whole-plane SLE$(\kappa;\rho)$ process.

Let $\kappa>0$ and and $\rho\in\R$. Let $\sigma\in\{1,-1\}$. The case $\sigma=1$ (resp.\ $=-1$) corresponds to the forward (resp.\ backward) process. 
Let $z\ne w\in\TT$. Choose $x,y\in\R$ such that $e^{ix}=z$, $e^{iy}=w$, and $0<x-y<2\pi$.  Let $\lambda(t)$ and $q(t)$, $0\le t<T$, be the solution of the system of SDE:
\BGE \left\{\begin{array}{ll} d\lambda(t)=\sqrt\kappa dB(t)+\sigma\frac\rho2\cot_2(\lambda(t)-q(t))dt,& \lambda(0)=x;\\
dq(t)=\sigma\cot_2(q(t)-\lambda(t))dt,&q(0)=y.
\end{array}\right.\label{kappa-tau-sigma}\EDE
If $\sigma=1$ (resp.\ $=-1$), the forward (resp.\ backward) radial Loewner process driven by $\lambda$ is called a forward (resp.\ backward) SLE$(\kappa;\rho)$ process started from $(z;w)$. Recall that $\cot_2(z)=\cot(z/2)$. The appearance of $\cot_2$ comes from the covering forward and backward radial Loewner equations. Since $\cot_2$ has period $2\pi$,  it is easy to see that the definition does not depend on the choice of $x,y$.

Let $Z_t=\lambda(t)-q(t)$. Then $(\frac 12Z_{\frac 4\kappa t})$ is a radial Bessel process of dimension $\delta:=\frac 4\kappa \sigma(\frac\rho 2+1)+1$ (See Appendix \ref{B'}). Thus, $T=\infty$ if $\delta\ge 2$; $T<\infty$ if $\delta<2$.

\begin{Lemma}
  Let $\kappa>0$ and $\rho\le -\frac\kappa 2-2$. Let $L_t$, $0\le t<\infty$, be $\D$-hulls generated by a backward radial SLE$(\kappa;\rho)$ process started from $(z;w)$. Then $\bigcup_{t\ge 0} S_{L_t}=\TT\sem\{w\}$. \label{joint}
\end{Lemma}
\begin{proof} Since $\sigma=-1$ for the backward equation, $\rho\le -\frac\kappa 2-2$ implies that $\delta\ge 2$, and so $T=\infty$.
Let $f_t$, $0\le t<\infty$, be the conformal maps generated by the backward radial SLE$(\kappa;\rho)$ process. Formula (\ref{kappa-tau-sigma}) in the case $\sigma=-1$ implies that $e^{iq(t)}=f_t(w)$, $0\le t<\infty$. This means that $w\not\in S_{L_t}$, $0\le t<\infty$. The conclusion then follows from Lemma \ref{union-support}.
\end{proof}

Assume that $\delta\ge 2$, which means that $\rho\ge \frac\kappa 2-2$ if $\sigma=1$ and $\rho\le -\frac\kappa 2-2$ if $\sigma=-1$. From Corollary \ref{Xtstationary'}, $(Z_t)$ has a unique stationary distribution $\mu_\delta$, which has a density proportional to $\sin_2(x)^{\delta-1}$, and the stationary process is reversible. Let $(\bar Z_t)_{t\in\R}$ denote the stationary process. Let $\bar y$ be a random variable with uniform distribution $U_{[0,2\pi)}$ on $[0,2\pi)$ such that $\bar y$ is independent of $(\bar Z_t)$. Let $\bar q(t)={\bar y}-\sigma\int_0^t \cot_2(\bar Z_s)ds$ and $\bar \lambda(t)=\bar q(t)+\bar Z_t$, $t\in\R$. If $\sigma=1$ (resp.\ $=-1$), the forward (resp.\ backward) radial Loewner process driven by $\bar \lambda(t)$, $0\le t<\infty$, is called a stationary forward (resp.\ backward) radial SLE$(\kappa;\rho)$ process. Equivalently, a stationary forward (resp.\ backward) radial SLE$(\kappa;\rho)$ process is a forward (resp.\ backward) radial SLE$(\kappa;\rho)$ process started from a random pair $(e^{i\bar x},e^{i\bar y})$ with $(\bar x,\bar x-\bar y)\sim U_{[0,2\pi)}\times \mu_\delta$. If $\sigma=1$, the whole-plane Loewner process driven by $\bar \lambda(t)$, $t\in\R$, is called a whole-plane SLE$(\kappa;\rho)$ process.

It is easy to verify the following Markov-type relation between a whole-plane SLE$(\kappa;\rho)$ process and a forward radial SLE$(\kappa;\rho)$ process. Recall that $I_\TT(z)=1/\bar z$ is the reflection about $\TT$.
Let $g_t$ and $K_t$, $t\in\R$, be maps and hulls generated by a whole-plane SLE$(\kappa;\rho)$ process. Let $t_0\in\R$. Then $I_\TT\circ g_{t_0+t}\circ g_{t_0}^{-1}\circ I_\TT$ and $I_\TT\circ g_{t_0}(K_{t_0+t}\sem K_{t_0})$, $t\ge 0$, are maps and hulls generated by a stationary forward radial SLE$(\kappa;\rho)$ process.

Using the reversibility of the stationary radial Bessel processes of dimension $\delta\ge 2$, we obtain the following lemma.

\begin{Lemma}
  Let $\kappa>0$ and $\rho\le -\frac\kappa 2-2$. Let $\lambda(t)$, $t\ge 0$, be a driving function of a stationary backward radial SLE$(\kappa;\rho)$ process. Then for any $t_0>0$, $\lambda(t_0-t)$, $0\le t\le t_0$, is a driving function up to time $t_0$ of a stationary forward radial SLE$(\kappa;-4-\rho)$ process; and $\lambda(-t)$, $-\infty< t\le 0$, is a driving function up to time $0$ of a whole-plane SLE$(\kappa;-4-\rho)$ process. \label{reversal-driving'}
\end{Lemma}

Girsanov's theorem implies that many properties of forward or backward radial SLE$_\kappa$ process carry over to radial SLE$(\kappa;\rho)$ processes. For example, a forward (resp.\ backward) radial SLE$(\kappa;\rho)$ process generates a forward radial trace (resp.\ a family of backward radial traces). 
If $\kappa\le 4$ and $\rho\le -\frac\kappa 2-2$, then a backward radial SLE$(\kappa;\rho)$ process induces a welding, say $\phi$, of $\TT$ with two fixed points. Suppose the process is started from $(z;w)$. From $e^{i\lambda(0)}=e^{i(q(0)+Z_0)}=e^{ix}=z$ we see that $z$ is one fixed point of $\phi$. Lemma \ref{joint} implies that $w$ is the joint point of the process, and so is the other fixed point of $\phi$.

\begin{Corollary}
  Let $\kappa>0$ and $\rho\le -\frac\kappa 2-2$. Let $(\beta_t)$ be a family of backward radial traces generated by a stationary backward radial SLE$(\kappa;\rho)$ process. Let $\beta$ be a stationary forward radial SLE$(\kappa;-4-\rho)$ trace. Then for every fixed $t_0\in(0,\infty)$, $\beta_{t_0}(t)$, $0\le t\le t_0$, has the same distribution as $\beta(t_0-t)$, $0\le t\le t_0$. \label{reversal-trace}
\end{Corollary}

\no{\bf Remark.} One special value of $\rho$ is $-4$. Theorem 6.8 in \cite{RZ} implies that, if $\kappa\in(0,4]$, a stationary backward radial SLE$(\kappa;-4)$ process is a stationary  backward radial SLE$_\kappa$ process, i.e., the process driven by $\lambda(t)=\bar x+\sqrt{\kappa} B(t)$, where $\bar x$ is a random variable uniformly distributed on $[0,2\pi)$ and independent of $B(t)$. So the above corollary provides a connection between a family of stationary backward radial SLE$_\kappa$ traces and a stationary forward radial SLE$_\kappa$ trace.

\vskip 4mm
We are especially interested in the backward radial SLE$(\kappa;-\kappa-6)$ processes. The proposition below is a particular case of Theorem 4.7 in \cite{RZ}.

\begin{Proposition}
  Let $\kappa>0$ and $z_0\ne z_\infty\in\TT$. Let $f_t$ and $L_t$, $0\le t<\infty$, be the backward radial SLE$(\kappa;-\kappa-6)$ maps and hulls started from $(z_0,z_\infty)$. Let $W$ be a M\"obius transformation with $W(\D)=\HH$, $W(z_0)=0$, and $W(z_\infty)=\infty$. Then there is a strictly increasing function $v$ with $v([0,\infty))=[0,\infty)$ such that $W^*(L_{v(t)})$, $0\le t<\infty$, are the $\HH$-hulls driven by a backward chordal SLE$_\kappa$ process. \label{Prop1}
\end{Proposition}

The symbol $W^*(L)$ is defined in Section 2.3 of \cite{RZ}. Lemma 2.20 in \cite{RZ} ensures  that for a $\D$-hull $L$ and a M\"obius transformation $W$ with $W^{-1}(\infty)\not\in S_L$, there is a unique M\"obius transformation $W^L$ with $W^L(\D)=\HH$ such that $W^L(L)$ is an $\HH$-hull, and $W^L\circ f^\D_L=f^{\HH}_{W^L(L)}\circ W$
holds in $\D$. 
The $W^*(L)$ is defined to be the $\HH$-hull $W^L(L)$. Since $z_\infty$ is the joint point of the process, $W^{-1}(\infty)=z_\infty\not\in S_{L_t}$ for each $t$, and so $W^{L_t}$ and $W^*(L_t)$ are well defined.

Write $W_t=W^{L_t}$, $0\le t< \infty$. Let $\lambda$ be the driving function for the backward radial Loewner process $(L_t)$. Let $\ha\lambda$ be the driving function for the backward chordal process $(W^*(L_{v(t)})=W_{v(t)}(L_{v(t)}))$. Then (4.10) in \cite{RZ} implies that $W_t(e^{i\lambda(t)})=\ha\lambda(v(t))$. In fact, in (4.10) of \cite{RZ}, the $\til W$ satisfies that $e^{i\til W(z)}=W(e^{iz})$, and the $\lambda^*(t)$ corresponds to the $\ha\lambda(v(t))$ here. Let $f_t$ (resp.\ $\ha f_t$), $f_{t_2,t_1}$ (resp.\ $\ha f_{t_2,t_1}$), and $(\beta_t)$ (resp.\ $\ha\beta_t$), $0\le t<\infty$, be the backward radial (resp.\ chordal) Loewner maps, flows, and traces driven by $\lambda$ (resp.\ $\ha\lambda$).
Then we have $W_t\circ f_t=\ha f_{v^{-1}(t)}\circ W$ in $\D$ for any $t\ge 0$. Applying this equality to $t=t_2$ and $t=t_1$, where $t_2\ge t_1\ge 0$, and using Lemma \ref{ft2t1}, we get
$W_{t_2}\circ f_{t_2,t_1}\circ f_{t_1}=\ha f_{v^{-1}(t_2),v^{-1}(t_1)}\circ W_{t_1}\circ f_{t_1}$
in $\D$, which implies that $W_{t_2}\circ f_{t_2,t_1}=\ha f_{v^{-1}(t_2),v^{-1}(t_1)}\circ W_{t_1}$ in $\D$, and so
$$\ha\beta_{t_2}(t_1)=\ha f_{t_2,t_1}(\ha \lambda(t_1))=\ha f_{t_2,t_1}\circ W_{v(t_1)}(e^{i\lambda(v(t_1))})$$$$=W_{v(t_2)}\circ f_{v(t_2),v(t_1)}(e^{i\lambda(v(t_1))})=W_{v(t_2)}(\beta_{v(t_2)}(v(t_1))).$$
Thus, the proposition above implies the following corollary.

\begin{Corollary}
  Let $\kappa>0$ and $z_0\ne z_\infty\in\TT$. Let $\beta_t$, $0\le t<\infty$, be the backward radial SLE$(\kappa;-\kappa-6)$ traces started from $(z_0,z_\infty)$. Then there exist a strictly increasing function $v$ with $v([0,\infty))=[0,\infty)$, and a family of M\"obius transformations $(W_t)_{t\ge 0}$ with $W_t(\D)=\HH$, such that $\ha\beta_{t}:=W_{v(t)}\circ \beta_{v(t)}\circ v$, $0\le t<\infty$, are backward chordal traces generated by a backward chordal SLE$_\kappa$ process. \label{Prop1-cor}
\end{Corollary}

The following proposition is Theorem 6.1 in \cite{RZ}.

\begin{Proposition}
  Let $\kappa\in(0,4]$. Let $z_1\ne z_2\in\TT$. There is a coupling of two backward radial SLE$(\kappa;-\kappa-6)$ processes, one started from $(z_1;z_2)$, the other started from $(z_2;z_1)$, such that the two processes induce the same welding. \label{Prop2}
\end{Proposition}


\no{\bf Remark.} If $\delta=\frac 4\kappa \sigma(\frac\rho 2+1)+1\in(1,2)$, we may define a forward (resp.\ backward) radial SLE$(\kappa;\rho)$ process in the case $\sigma=1$ (resp.\ $\sigma=-1$) such that the time interval of the process is $[0,\infty)$. First, the second remark in Appendix \ref{B'} says that a radial Bessel process $(X_t)$ of dimension $\delta>0$ started from $(x-y)/2$ can be defined for all $t\ge 0$. Second, the transition density of $(X_t)$ given by Proposition (\ref{densityY'}) (which is also true in the case $\delta\in(0,2)$) shows that, if $\delta>1$, then $\cot(X_t)$, $0\le t<\infty$, is locally integrable. Thus, if $\delta>1$, we may let $q(t)=y-\sigma\int_0^t \cot_2(Z_s)ds$ and $\lambda(t)=q(t)+Z_t$, $0\le t<\infty$, where $Z_t=2X_{\frac \kappa 4 t}$, and use $\lambda$ as the driving function to define a forward (resp.\ backward) radial SLE$(\kappa;\rho)$ process. The corresponding stationary processes are similarly defined. Lemma \ref{reversal-driving'} still holds thanks to the reversibility of the stationary radial Bessel process in the case $\delta\in (1,2)$. But Girsanov's theorem does not apply beyond the time that $\lambda(t)-q(t)$ hits $\{0,2\pi\}$.

\section{Normalized Backward Radial Loewner Trace} \label{normalized}
In general, a backward chordal (resp.\ radial) Loewner process does not naturally generate a single curve even if the backward chordal (resp.\ radial) traces $(\beta_t)$ exist, because they may not satisfy $\beta_{t_1}\subset\beta_{t_2}$ when $t_1\le t_2$. A normalization method was introduced in \cite{RZ} to define a normalized backward chordal Loewner trace (under certain conditions). In this section we will define a normalized backward radial Loewner trace.


\begin{Lemma} Let $\lambda\in C([0,\infty))$, and $(f_{t_2,t_1})$ be the backward radial Loewner flow driven by $\lambda$.
Define $F_{t_2,t_1}=e^{t_2} f_{t_2,t_1}$, $t_2\ge t_1\ge 0$. Then for every fixed $t_0\in[0,\infty)$, $F_{t,t_0}$ converges locally uniformly in $\D$ as $t\to\infty$ to a conformal map, denoted by $F_{\infty,t_0}$, which satisfies that $F_{\infty,t_0}(0)=0$, $F_{\infty,t_0}'(0)=e^{t_0}$, and
\BGE F_{\infty,t_2}\circ f_{t_2,t_1}=F_{\infty,t_1},\quad t_2\ge t_1\ge 0.\label{Ft2t1}\EDE
Moreover, let  $G_s=I_{\TT}\circ F_{\infty,-s}^{-1}\circ I_{\TT}$ and  $K_s=\C\sem I_\TT\circ F_{\infty,-s}(\D)$, $-\infty<s\le 0$. Then $G_s$ and $K_s$ are whole-plane Loewner maps and hulls driven by $\lambda(-s)$, $-\infty<s\le 0$.
  \label{lemma1}
\end{Lemma}
\begin{proof}
Lemma \ref{ft2t1} (ii) implies that, if $t_2\ge t_1\ge 0$, then $f_{t_2,t_1}$ is a conformal map on $\D$ with $f_{t_2,t_1}(0)=0$ and $f_{t_2,t_1}'(0)=e^{-(t_2-t_1)}$. Thus, every $F_{t_2,t_1}$ is a conformal map on $\D$ that satisfies $F_{t_2,t_1}(0)=0$ and $F_{t_2,t_1}'(0)=e^{t_1}$. Koebe's distortion theorem (c.f.\ \cite{Ahl}) implies that, for every fixed $t_1$, $(F_{t_2,t_1})_{t_2\ge t_1}$ is a normal family. Let $S$ be a countable unbounded subset of $[0,\infty)$, and write $S_{\ge t}=\{x\in S:x\ge t\}$ for every $t\ge 0$.
Using a diagonal argument, we can find a positive sequence $t_n\to \infty$ such that for any $x\in S$, $(F_{t_n,x})$ converges locally uniformly in $\D$. Let $F_{\infty,x}$ denote the limit. Lemma \ref{domain convergence*} implies that $F_{\infty,x}$ is a conformal map on $\D$, and satisfies $F_{\infty,x}(0)=0$ and $F_{\infty,x}'(0)=e^{x}$.

Let $x_2\ge x_1\in S$. From $f_{t_n,x_2}\circ f_{x_2,x_1}=f_{t_n,x_1}$ we conclude that $F_{\infty,x_2}\circ f_{x_2,x_1}=F_{\infty,x_1}$. For $t\in[0,\infty)$, choose $x\in S_{\ge t}$ and define the conformal map $F_{\infty,t}=F_{\infty,x}\circ f_{x,t}$ on $\D$. Lemma \ref{ft2t1} (i) and $F_{\infty,x_2}\circ f_{x_2,x_1}=F_{\infty,x_1}$ for $x_2\ge x_1\in S$ imply that the definition of $F_{\infty,t}$ does not depend on the choice of $x\in S_{\ge t}$, and (\ref{Ft2t1}) holds.

From (\ref{radial-backward-2}) we see that $f_{t_2,t_1}$ commutes with the reflection $I_\TT(z)=1/\bar z$. Since $f_{t_2,t_1}^{-1}=f_{t_1,t_2}$, using (\ref{Ft2t1}) we get $G_{s_1}=f_{-s_1,-s_2}\circ G_{s_2}$ if $s_1\le s_2\le 0$. From (\ref{radial-backward-2}) we see that $G_{s}$ satisfies the equation
\BGE \pa_s G_{s}(z)=G_{s}(z)\frac{e^{i\lambda(-s)}+G_{s}(z)}{e^{i\lambda(-s)}-G_{s}(z)},\quad -\infty<s\le 0.\label{paG}\EDE

Let $\ha F_{\infty,t}(z)=F_{\infty,t}(e^{-t} z)$, $t\ge 0$. Then each $\ha F_{\infty,t}$ is a conformal map defined on $e^t\D $, and satisfies $\ha F_{\infty,t}(0)=0$ and $\ha F_{\infty,t}'(0)=1$. As $t\to\infty$, $e^t\D\dto \C$ (c.f.\ Definition \ref{def-lim}). Koebe's distortion theorem implies that $|\ha F_{\infty,t}(z)|\le \frac{|z|}{(1-e^{-t}|z|)^2}$ for $z\in e^t\D$. Thus, for every $r>0$, there exists $t_0\in\R$ such that, if $t\ge t_0$, then $|\ha F_{\infty,t}|\le 2r$ on $\{|z|\le r\}$. Therefore, every sequence $(t_n)$, which tends to $\infty$, contains a subsequence $(t_{n_k})$ such that $\ha F_{\infty,t_{n_k}}$ converges locally uniformly in $\C$. Applying Lemma \ref{domain convergence*}, we see that the limit function is a conformal map on $\C$, which fixes $0$ and has derivative $1$ at $0$. Such conformal map must be the identity. Hence $\ha F_{\infty,t}\luto \id$ in $\C$ as $t\to\infty$. Applying Lemma \ref{domain convergence*} again, we see that $e^t F_{\infty,t}^{-1}(z)\luto \id$ in $\C$ as $t\to\infty$. Thus, $\lim_{s\to-\infty} e^{s} G_{s}(z)=z$ for any $z\in\C\sem\{0\}$, which together with (\ref{paG}) implies that $G_s$, $-\infty<s\le 0$, are whole-plane Loewner maps driven by $\lambda(-s)$. The $K_s$ are the corresponding hulls because $K_s=\C\sem G_s^{-1}(\D^*)$.

It remains to show that, for any $t\in[0,\infty)$, $F_{x,t}\luto F_{\infty,t}$ in $\D$ as $x\to \infty$. Assume that this is not true for some $t_0\in[0,\infty)$. Since $(F_{x,t_0})_{x\ge t_0}$ is a normal family, there exists $x_n\to\infty$ such that $F_{x_n,t_0}$ converges locally uniformly in $\D$ to a function other than $F_{\infty,t_0}$. Let $\til F_{\infty,t_0}$ denote the limit. Let $S=\N\cup\{t_0\}$. By passing to a subsequence, we may assume that, for every $t\in S$, $F_{x_n,t}\luto \til F_{\infty,t}$ in $\D$. Now we may repeat the above construction to define $\til F_{\infty,t}$ for every $t\in[0,\infty)$. The previous argument shows that $I_\TT\circ\til F_{\infty,-t}^{-1}\circ I_\TT$, $-\infty<t\le 0$, are the whole-plane Loewner maps driven by $\lambda(-t)$, $-\infty<t\le 0$. Since the same is true for $I_{\TT}\circ F_{\infty,-t}^{-1}\circ I_{\TT}$, we get $\til F_{\infty,t}=F_{\infty,t}$ for every $t$, which contradicts that $\til F_{\infty,t_0}\ne  F_{\infty,t_0}$. Thus, $F_{x,t}\luto F_{\infty,t}$ in $\D$ as $x\to \infty$.
\end{proof}

\begin{Lemma}
Let $\lambda\in C([0,\infty))$. Let $(F_{\infty,t})_{t\ge 0}$ be given by the above lemma. Suppose the backward radial Loewner process driven by $\lambda$ generates a family of backward radial Loewner traces $\beta_t$, $0\le t<\infty$, and
  \BGE \forall t_0\in[0,\infty),\quad \exists t_1\in(t_0,\infty),\quad \beta_{t_1}([0,t_0])\subset\D.  \label{beta>0}\EDE
Then every $F_{\infty,t}$ extends continuously (in the spherical metric) to $\lin\D$, and there is a continuous curve $\beta(t)$, $0\le t<\infty$, with $\lim_{t\to\infty}\beta(t)=\infty$ such that \BGE \beta(t)=F_{\infty,t_0}(\beta_{t_0}(t)),\quad t_0\ge t\ge 0;\label{beta1}\EDE
and for any $t\ge 0$, $F_{\infty,t}(\D)$ is the component of $\C\sem \beta([t,\infty))$ that contains $0$.
Furthermore, $\gamma(s):=I_\TT(\beta(-s))$, $-\infty<s\le 0$, is the whole-plane Loewner trace driven by $\lambda(-s)$. \label{lemma2}
\end{Lemma}
\begin{proof}
For every $t_0\in[0,\infty)$, using (\ref{beta>0}) we may pick $t_1\in(t,T)$ such that $\beta_{t_1}([0,t_0])\subset\D$, and define
$ \beta(t)=F_{\infty,t_1}\circ \beta_{t_1}(t)$,  $ t\in[0,t_0]$. From (\ref{backward-trace}) and (\ref{Ft2t1}) we see that the definition of $\beta$ does not depend on $t_0$ and $t_1$, and $\beta$ is continuous on $[0,\infty)$.

Let $L_{t_2,t_1}=\D\sem f_{t_2,t_1}(\D)$, $t_2\ge t_1\ge 0$. Then $L_{t_2,t_1}$ is the $\D$-hull generated by $ \beta_{t_2}([t_1,t_2])$, i.e., $\D\sem L_{t_2,t_1}$ is the component of $\D\sem \beta_{t_2}([t_1,t_2])$ that contains $0$. Hence $\pa L_{t_2,t_1}\cap\D\subset \beta_{t_2}([t_1,t_2])$.

Let $G_s$ and $K_s$, $-\infty<s\le 0$, be given by the previous lemma. Then $(K_s)$ is an increasing family with $\bigcap_{s\le 0} K_s=\{0\}$. If $s_2\le s_1\le 0$, from $F_{\infty,-s_1}=F_{\infty,-s_2}\circ f_{-s_2,-s_1}$ and $f_{-s_2,-s_1}(\D)=\D\sem L_{-s_2,-s_1}$, we see that $K_{s_1}\sem K_{s_2}=I_{\TT}\circ F_{\infty, -s_2}(L_{-s_2,-s_1})$.

Fix $t_2\ge t_1\ge 0$. Choose $T>t_2$ such that $\beta_T([0,t_2])\subset\D$. Then $\beta(t_2)=F_{\infty,T}\circ \beta_T(t_2)$. Since $f_{T,t_1}:\D\conf \D\sem L_{T,t_1}$, $L_{T,t_1}$ is the $\D$-hull generated by $\beta_T([t_1,T])$, and $t_2\in[t_1,T]$, we see that $\beta_T(t_2)\not\in f_{T,t_1}(\D)$. So $\beta(t_2)\not\in F_{\infty,T}\circ f_{T,t_1}(\D)=F_{\infty,t_1}(\D)$. This implies that, if $s_2\le s_1\le 0$, then $\gamma(s_2)=I_\TT(\beta(-s_2))\in\C\sem I_\TT \circ F_{\infty,-s_1}(\D)=K_{s_1}$. Thus, $\gamma((-\infty,s])\subset K_s$ for every $s\le 0$. Since $\bigcap_{s\le 0} K_s=\{0\}$, we get $\lim_{s\to-\infty}\gamma(s)=0$.

Define $\gamma(-\infty)=0$. Let $s\le 0$. 
Let $z_0\in\pa K_s$. If $z_0=0$, then $z_0=\gamma(-\infty)\in\gamma([-\infty,s])$. Now suppose $z_0\ne 0$. Since $(K_s)$ is increasing and $\bigcap_{s\le 0} K_s=\{0\}$, there is $s_0<s$ such that $z_0\not\in K_{s_0}$. Thus, $z_0\in K_s\sem K_{s_0}=I_{\TT}\circ F_{\infty,-s_0}(L_{-s_0,-s})$. From $z_0\in\pa K_s$ we see that $w_0:=F_{\infty,-s_0}^{-1}\circ I_\TT(z_0)\in \pa L_{-s_0,-s}\cap \D$. Since $L_{-s_0,-s}$ is the $\D$-hull generated by $\beta_{-s_0}([-s,-s_0])$, there is $t_1\in[-s,-s_0]$ such that $w_0=\beta_{-s_0}(t_1)$. Thus, $z_0=I_\TT\circ F_{\infty,-s_0}(\beta_{-s_0}(t_1))=\gamma(-t_1)\in\gamma ([-\infty,s])$. Thus, $\pa K_s\subset \gamma([-\infty,s])$, which implies that $\pa K_s$ is locally connected. Since $I_\TT\circ F_{\infty,-s}\circ I_\TT:\D^*\conf \ha\C\sem K_s$, we see that $F_{\infty,t}$ extends continuously to $\lin\D$ for each $t\ge 0$ (c.f.\ \cite{Pom-bond}). The equality (\ref{Ft2t1}) holds after continuation, which together with (\ref{backward-trace}) and the definition of $\beta$ implies (\ref{beta1}). Setting $t_1=t=-s$, we see that $\gamma(s)=I_\TT\circ F_{\infty, t}(e^{i\lambda(t)})=G_s^{-1}(e^{i\lambda(-s)})$. Thus, $\gamma(s)$, $-\infty\leq s\le 0$, is the whole-plane Loewner trace driven by $\lambda(-s)$, $-\infty<s\le 0$. This implies that $\lim_{t\to\infty}\beta(t)=I_\TT(\lim_{s\to-\infty}\gamma(s))=\infty$.

Finally, from the properties of the whole-plane Loewner trace, we see that for any $s\ge 0$, $G_s^{-1}(\D^*)$ is the component of $\ha\C\sem \gamma([-\infty,s])$ that contains $0$. Since $G_{-t}=I_\TT\circ F_{\infty,t}^{-1}\circ I_\TT$ and $\gamma(-t)=I_\TT(\beta(t))$, we see that, for any $t\ge 0$, $F_{\infty, t}(\D)$ is the component of $\C\sem \beta([t,\infty))$ that contains $I_\TT(\infty)=0$.
\end{proof}

\begin{Definition}
  The  $\beta(t)$, $0\le t<\infty$, given by the lemma is called the normalized backward radial Loewner trace driven by $\lambda$
\end{Definition}

If the backward radial Loewner traces $\beta_t$ are all $\D$-simple traces, then (\ref{beta>0}) clearly holds because we may always choose $t_1=t_0+1$. Moreover, (\ref{beta1}) implies that for any $t_0>0$, $\beta$ restricted to $[0,t_0)$ is simple. Thus, the whole curve $\beta$  is simple. This implies further that $F_{\infty,t}(\D)=\C\sem\beta([t,\infty))$ for any $t\ge 0$. In particular, $F_{\infty,0}$ maps two arcs on $\TT$ with two common end points onto the two sides of $\beta$. Let $\phi$ be the welding induced by the process. The equality $F_{\infty,0}=F_{\infty,t}\circ f_t$ implies that, if $y=\phi(x)$ then $F_{\infty,0}(x)=F_{\infty,0}(y)\in\beta$. The two fixed points of $\phi$ are mapped to the two ends of $\beta$ such that $e^{i\lambda(0)}$ is mapped to $\beta(0)\in\C$, and the joint point is mapped to $\infty$.

We will prove that (\ref{beta>0}) holds in some other cases. We say that an $\HH$(resp.\ $\D$)-hull $K$ is nice if $S_K$ is an interval on $\R$(resp.\ $\D$), and $f_K$ extends continuously to $S_K$ and maps the interior of $S_K$ into $\HH$ (resp.\ $\D$). This means that $\pa K\cap\HH$ (resp.\ $\pa K\cap\D$) is the image of an open curve in $\HH$ (resp.\ $\D$), whose two ends approach $\R$ (resp.\ $\TT$). It is easy to see that, if $K$ is a nice $\HH$-hull, and $W$ is a M\"obius transformation such that $W(\HH)=\D$ and $0\not\in W(K)$, then $W(K)$ is a nice $\D$-hull.

\begin{Lemma}
  Let $\kappa>4$ and $\rho\le -\frac\kappa 2-2$. Let $(L_t)$ be $\D$-hulls generated by a backward radial SLE$(\kappa;\rho)$ process. Then for every fixed $t_0\in(0,\infty)$, a.s.\ $L_{t_0}$ is nice. \label{nice}
\end{Lemma}
\begin{proof}
Theorem 6.1 in \cite{duality2} shows that, if $(H_t)$ are $\HH$-hulls generated by a (forward) chordal SLE$_\kappa$ process, then for any stopping time $T\in(0,\infty)$, a.s.\ $H_T$ is a nice $\HH$-hull. From the equivalence between chordal SLE$_\kappa$ and radial SLE$_\kappa$ (Proposition 4.2 in \cite{LSW2}), we conclude that, if $(K_t)$ are $\D$-hulls generated by a forward radial SLE$_\kappa$ process, then for any deterministic point $z_0\in\TT$ and any stopping time $T\in(0,\infty)$ such that $z_0\not\in\lin{K_T}$, a.s.\ $K_T$ is a nice $\D$-hull. This further implies that, for any stopping time $T\in(0,\infty)$, on the event that $\TT\not\subset\lin {K_T}$, a.s.\ $K_T$ is a nice $\D$-hull.
Let $(L^0_t)$ be $\HH$-hulls generated by a backward radial SLE$_\kappa$ process.
The above result in the case that $T$ is a deterministic time together with Lemma \ref{ft2t1} and the rotation symmetry of radial Loewner processes implies that,  for any fixed $t_0\in (0,\infty)$, on the event that $S_{L^0_{t_0}}\ne \TT$, a.s.\ $L^0_{t_0}$ is a nice $\D$-hull.

By rotation symmetry, we may assume that the backward radial SLE$(\kappa;\rho)$ process which generates $(L_t)$ is started from $(1;w_0)$.
Fix $t_0\in(0,\infty)$. Girsanov's theorem implies that the distribution of $(L_t)_{0\le t\le t_0}$ is absolutely continuous w.r.t.\ that of $(L^0_t)_{0\le t\le t_0}$ given by the last paragraph conditioned on the event that $f^0_t(w_0)\in\TT$ for $0\le t\le t_0$. Since $f^0_{t_0}(w_0)\in\TT$ is equivalent to $w_0\in\TT\sem S_{L_{t_0}}$, which implies that $S_{L_{t_0}}\ne \TT$, the proof is completed.
\end{proof}

\begin{Proposition}
  Let $\kappa>0$ and $\rho\le -\frac\kappa 2-2$. Then condition (\ref{beta>0}) almost surely holds for a backward radial SLE$(\kappa;\rho)$ process. \label{condition}
\end{Proposition}
\begin{proof}
The result is clear if $\kappa\le 4$ since the traces are $\D$-simple. Now assume that $\kappa>4$.
  Suppose the process is started from $(z_0;w_0)$. Lemma \ref{joint} implies that $S_{L_{t_0}}\subset \TT\sem\{w_0\}$. So $f_{t_0}(w_0)\not\in \lin{L_{t_0}}$. Since $L_{t_0}$ is the $\D$-hull generated by $\beta_{t_0}$, we have $f_{t_0}(w_0)\not\in \beta_{t_0}([0,t_0])$. The Markov property of Brownian motion and the fact that $e^{i q(t)}=f_t(w_0)$ for all $t$ imply that, conditioned on $\lambda(t)$, $0\le t\le t_0$, the maps $f_{t_0+t,t_0}$, $t\ge 0$, are generated by a backward radial SLE$(\kappa;\rho)$ process started from $(e^{i\lambda(t_0)};f_{t_0}(w_0))$. Let $L_{t_0+t,t_0}=\D\sem f_{L_{t_0+t,t_0}}(\D)$. Lemma \ref{nice} implies that, for every $t_1>t_0$, a.s.\ $L_{t_1,t_0}$ is nice. Lemma \ref{joint} implies that the probability that $\beta_{t_0}([0,t_0])\cap\TT$ is contained in the interior of $S_{L_{t_1,t_0}}$ tends to $1$ as $t_1\to\infty$.

  If $L_{t_1,t_0}$ is nice and $\beta_{t_0}([0,t_0])\cap\TT$ is contained in the interior of $S_{L_{t_1,t_0}}$, then
  $$\beta_{t_1}([0,t_0])=f_{t_1,t_0}(\beta_{t_0}([0,t_0]))=f_{L_{t_1,t_0}}(\beta_{t_0}([0,t_0]))\subset\D.$$
  In fact, if $z\in \beta_{t_1}([0,t_0])\cap\D$, then obviously $f_{L_{t_1,t_0}}(z)\in\D$; if $z\in \beta_{t_0}([0,t_0])\cap\TT$, then $f_{L_{t_1,t_0}}(z)\in\D$ follows from that $L_{t_1,t_0}$ is nice and $z$ lies in the interior of $S_{L_{t_1,t_0}}$. Thus, as $t_1\to\infty$, the probability that $\beta_{t_1}([0,t_0])\subset\D$ tends to $1$. This means that, for every fixed $t_0>0$, a.s.\ there exists a (random) $t_1>t_0$ such that $\beta_{t_1}([0,t_0])\subset\D$. Thus, on an event with probability $1$, (\ref{beta>0}) holds for every $t_0\in\N$. Since $\beta_{t_1}([0,t_0])\subset \beta_{t_1}([0,n])\subset\D$ if $t_0<n\in\N$, we see that (\ref{beta>0}) holds on that event. This completes the proof.
\end{proof}

Thus, a normalized backward radial SLE$(\kappa;\rho)$ trace can be well defined for any $\kappa>0$ and $\rho\le -\frac\kappa 2-2$. Combining Lemmas \ref{reversal-driving'} and \ref{lemma2}, we obtain the following theorem.

\begin{Theorem}
  Let $\kappa>0$ and $\rho\le -\frac\kappa 2-2$. Let $\beta(t)$, $0\le t<\infty$, be a normalized stationary backward radial SLE$(\kappa;\rho)$ trace. Then $\gamma(s):=I_\TT(\beta(-s))$, $-\infty<s\le 0$, is a whole-plane SLE$(\kappa;-4-\rho)$ trace stopped at time $0$. \label{cor-whole}
\end{Theorem}



\section{Conformal Images of the Tips} \label{image}
\begin{Theorem}
  Let $\kappa\in(0,4)$. Let $\gamma(s)$, $-\infty\le s\le 0$, be a whole-plane SLE$(\kappa;\kappa+2)$ trace stopped at time $0$. Then after an orientation reversing time-change, the curve $\gamma(s)-\gamma(0)$, $-\infty\le s\le 0$, has the same distribution as $\gamma(s)$, $-\infty\le s\le 0$. \label{reversibility-whole}
\end{Theorem}
\begin{proof}
  Theorem \ref{cor-whole} shows that $\beta(t):=I_\TT(\gamma(-t))$, $0\le t\le \infty$, is a normalized stationary backward radial SLE$(\kappa;-\kappa-6)$ trace, which is a simple curve with $ \beta(\infty)=\infty$, and there is $F_{\infty,0}:\D\conf \C\sem\beta$ such that $F_{\infty,0}(0)=0$, $F_{\infty,0}'(0)=1$, and $F_{\infty,0}(x)=F_{\infty,0}(y)$ implies that $y=x$ or $y=\phi(x)$, where $\phi$ is the welding welding induced by the stationary backward radial SLE$(\kappa;-\kappa-6)$ process. Proposition \ref{Prop2} implies that this process can be with another stationary backward radial SLE$(\kappa;-\kappa-6)$ process, which induces the same lamination, but has a different joint point. Let $\til\beta$ and $\til F_{\infty,0}$ be the normalized trace and map for the second process. Let $\til\gamma(s)=I_\TT(\til\beta(-s))$, $-\infty\le s\le 0$. Theorem \ref{cor-whole} implies that $\til\gamma$ is also a whole-plane SLE$(\kappa;\kappa+2)$ trace stopped at time $0$.

  Define $W=I_\TT\circ \til F_{\infty,0}\circ F_{\infty,0}^{-1}\circ I_\TT$. Then $W:\ha\C\sem \gamma\conf \ha\C\sem \til\gamma$ and satisfies that $W(\infty)=\infty$ and $W'(\infty)=1$. Since the two backward radial SLE$(\kappa;-\kappa-6)$ processes induce the same welding, we see that $F_{\infty,0}(x)=F_{\infty,0}(y)$ iff $\til F_{\infty,0}(x)=\til F_{\infty,0}(y)$. Thus, $W$ extends continuously to $\gamma$. The work in \cite{JS} shows that a H\"older curve is conformally removable; while the work in \cite{RS-basic} shows that, for $\kappa\in(0,4)$, a chordal SLE$_\kappa$ trace is a H\"older curve, which together with the Girsanov's theorem and the equivalence between chordal SLE$_\kappa$ and radial SLE$_\kappa$ implies that a radial SLE$(\kappa;\rho)$ trace is conformally removable for $\kappa\in(0,4)$ and $\rho\ge \frac\kappa 2-2$ (which is true if $\rho=\kappa+2$). The Markov-type relation between whole-plane SLE$(\kappa;\kappa+2)$ and radial SLE$(\kappa;\kappa+2)$ processes implies that $\gamma([t_0,0])$ is conformally removable for any $t_0\in(-\infty,0)$, and so is the whole curve $\gamma=\gamma([-\infty,0])$. Thus, $W$ extends to a conformal map defined on $\ha\C$ such that $W(\gamma)=\til\gamma$. Since $W(\infty)=\infty$ and $W'(\infty)=1$, we have $W(z)=z+C$ for some constant $C\in\C$. This means that $\til\gamma=\gamma+C$, where both curves are viewed as sets. Since both curves are simple, $W$ maps end points of $\gamma$ to end points of $\til\gamma$. Now $0$ is an end point of both curves. Since $F_{\infty,0}$ and $\til F_{\infty,0}$ map the joint points of the two processes, respectively, to $\infty$, and the two joints points are different, $W$ does not fixed $0$. So $W$ maps the other end point of $\gamma$: $\gamma(0)$ to $0$, which implies that $C=-\gamma(0)$ and the orientations of $\til\gamma$ and $W(\gamma)=\gamma-\gamma(0)$ are opposite to each other. Thus, the whole-plane SLE$(\kappa;\kappa+2)$ trace $\til\gamma$ up to time $0$ is an orientation reversing time-change of $\gamma-\gamma(0)$ up to time $0$, which completes the proof.
  \end{proof}

\no{\bf Remark.}
This theorem says that a whole-plane SLE$(\kappa;\kappa+2)$ ($\kappa\in(0,4)$) trace stopped at whole-plane capacity time $0$ satisfies reversibility. So a tip segment of the trace at time $0$ has the same shape as an initial segment of the trace.

\begin{Lemma}
  Let $\kappa>0$. Let $\beta$ be a forward chordal SLE$_\kappa$ trace. Let $t_0\in(0,\infty)$ be fixed. Then there is a whole-plane SLE$(\kappa;\kappa+2)$ process, which generates hulls $(K_s)$ and a trace $\gamma$, and a random conformal map $W$ defined on $\HH$ such that $W(\HH)=\ha\C\sem K_{s_0}$ for some random $s_0<0$ and $W(\beta(t))= \gamma(v(t))$, $0\le t\le t_0$, where $v$ is a random strictly increasing function with $v([0,t_0])=[s_0,0]$. \label{W-trace}
\end{Lemma}
\begin{proof}
Let $\lambda$ be the driving function for $\beta$. Lemma \ref{ft2t1} and the translation symmetry implies that there is a backward chordal SLE$_\kappa$ process, which generates backward chordal traces $(\til\beta_t)$ such that $\til\beta_{t_0}(t_0-t)=\beta(t)-\lambda(t_0)$, $0\le t\le t_0$. Corollary \ref{Prop1-cor} implies that there exist a stationary backward radial SLE$(\kappa;-\kappa-6)$ process generating backward radial traces $(\ha\beta_t)$, a family of  M\"obius transformations $(V_t)$ with $V_t(\HH)=\D$ for each $t$,
 and a strictly increasing function $u$ with $u([0,\infty))=[0,\infty)$, such that $V_{t_1}(\til \beta_{t_1}(t))=\ha\beta_{u(t_1)}(u(t))$ for any  $t_1\ge t\ge 0$. In particular, it follows that $V_{t_0}(\beta(t)-\lambda(t_0))=\ha\beta_{u(t_0)}(u(t_0-t))$, $0\le t\le t_0$.

Let $\ha\beta$ be the normalized backward radial trace generated by that stationary backward radial SLE$(\kappa;-\kappa-6)$ process, which exists thanks to Proposition \ref{condition}. Lemmas \ref{lemma1} and \ref{lemma2} state that there exists a family of conformal maps $F_{\infty,t}$, $t\ge 0$, defined on $\D$, with continuation to $\lin\D$, such that $\ha\beta(t)=F_{\infty,t_1}(\beta_{t_1}(t))$ for any $t_1\ge t\ge 0$. In particular, we have
$$F_{\infty,u(t_0)}(V_{t_0}(\beta(t)-\lambda(t_0)))=F_{\infty,u(t_0)}(\ha\beta_{u(t_0)}(u(t_0-t)))=\ha\beta(u(t_0-t)),\quad 0\le t\le t_0.$$
Theorem \ref{cor-whole} states that $\gamma(s):=I_\TT(\ha\beta(-s))$, $-\infty<s\le 0$, is a whole-plane SLE$(\kappa;\kappa+2)$ trace stopped at time $0$. Lemma \ref{lemma1} states that $K_s:=\C\sem I_\TT\circ F_{\infty,-s}(\D)$ are the corresponding hulls.
Then we have $I_\TT(F_{\infty,u(t_0)}(V_{t_0}(\beta(t)-\lambda(t_0))))=\gamma(-u(t_0-t))$ , $0\le t\le t_0$.
Now it is easy to check that $W(z):=I_\TT(F_{\infty,u(t_0)}(V_{t_0}(z-\lambda(t_0))))$, $v(t):=-u(t_0-t)$, and $s_0:=-u(t_0)$ satisfy the desired properties.
\end{proof}

\begin{Theorem}
    Let $\kappa\in(0,4)$ and $t_0\in(0,\infty)$. Let $\beta(t)$, $t\ge 0$, be a forward chordal SLE$_\kappa$ trace (parameterized by the half-plane capacity). Then there is a random conformal map $V$ defined on $\HH$ such that $V(\beta(t_0))=0$, and $V(\beta(t_0-t))$, $0\le t\le t_0$, is an initial segment of a whole-plane SLE$(\kappa;\kappa+2)$ trace, up to a time-change.\label{Thm2}
\end{Theorem}
\begin{proof}
Lemma \ref{W-trace} states that we can map $\beta(t_0-t)$, $0\le t\le t_0$, conformally to a tip segment of a whole-plane SLE$(\kappa;\kappa+2)$ trace at time $0$. Then we may apply Theorem \ref{reversibility-whole}.
\end{proof}

We may derive a similar but weaker result for radial SLE.

\begin{Theorem}
  Let $\kappa\in(0,4)$ and $t_0\in(0,\infty)$. Let $\beta(t)$, $t\ge 0$, be a forward radial SLE$_\kappa$ trace (parameterized by the disc capacity). Then there is a random conformal map $V$ defined on $\D$ such that $V(\beta(1))=0$, and up to a time-change, $V(\beta(t_0-t))$, $0\le t\le t_0$, has a distribution, which is absolutely continuous w.r.t.\ an initial segment of a whole-plane SLE$(\kappa;\kappa+2)$ trace.\label{Thm-R}
\end{Theorem}
\begin{proof}
From Theorem \ref{reversibility-whole}, it suffices to prove the theorem with ``an initial segment'' replaced by ``a tip segment at time $0$''.
  By rotation symmetry, we may assume that $\beta$ is a forward stationary radial SLE$(\kappa;0)$ trace. By Corollary \ref{reversal-trace}, $\beta(t_0-t)$, $0\le t\le t_0$, has the distribution of a backward stationary radial SLE$(\kappa;-4)$ trace at time $t_0$, say $\til\beta_{t_0}$. Girsanov's theorem implies that the distribution of $\til\beta_{t_0}$ is absolutely continuous w.r.t.\ a backward stationary radial SLE$(\kappa;-\kappa-6)$ trace at time $t_0$. This backward stationary radial SLE$(\kappa;-\kappa-6)$ trace at $t_0$ can then be mapped conformally to a tip segment of the normalized trace generated by the process. Finally, the reflection $I_\TT$ maps that tip segment to a tip segment of a whole-plane SLE$(\kappa;\kappa+2)$ trace at time $0$ thanks to Theorem \ref{cor-whole}.
\end{proof}

\section{Ergodicity}\label{ergodicity}
We will apply Theorems \ref{Thm2} and \ref{Thm-R} to study some ergodic behavior of the tip of a chordal or radial  SLE$_\kappa$ ($\kappa\in(0,4)$) trace at a deterministic half plane or disc capacity time.

Let $\gamma(t)$, $a\le t\le b$, be a simple curve in $\C$ such that $\gamma(a)=0$. We may reparameterize $\gamma$ using the whole-plane capacity. Let $T=\ccap(\gamma)$. Define $v$ on $[a,b]$ such that $v(a)=-\infty$ and $v(t)=\ccap(\gamma([a,t]))$, $a<t\le b$. Then $v$ is a strictly increasing function with $v([a,b])=[-\infty,T]$. It turns out that (c.f.\ \cite{LawSLE}) $\gamma^v(t):=\gamma(v^{-1}(t))$, $-\infty\le t\le T$, is a whole-plane Loewner trace driven by some $\lambda\in C((-\infty,T])$. Let $g_t$, $-\infty<t\le T$, be the corresponding maps. Then each $g_t^{-1}$ extends continuously to $\lin{\D^*}$ and maps $\TT$ onto $\gamma^v([-\infty,t])$. At time $t$, there are two special points on $\TT$, which are mapped by $g_t^{-1}$ to the two ends of $\gamma^v([-\infty,t])$. One is $e^{i\lambda(t)}$, which is mapped to $\gamma^v(t)$. Let $z(t)$ denote the point on $\TT$ which is mapped to $\gamma^v(-\infty)=0$. Then $z(t)$ satisfies the equation
$z'(t)=z(t)\frac{e^{i\lambda(t)}+z(t)}{e^{i\lambda(t)}-z(t)}$, $-\infty<t\le T$.
There exists a unique $q\in C((-\infty,T])$ such that $z(t)=e^{iq(t)}$ and $0<\lambda(t)-q(t)<2\pi$, $-\infty<t\le T$. Then $q(t)$ satisfies the equation $q'(t)=\cot_2(q(t)-\lambda(t))$, $-\infty<t\le T$.
 The number $\lambda(t)-q(t)\in(0,2\pi)$ has a geometric meaning. It is equal to $2\pi$ times the harmonic measure viewed from $\infty$ of the {\it right} side of $\gamma^v([-\infty,t])$ in $\ha\C\sem \gamma^v([-\infty,t])$.

Let $\kappa\le 4$ and $\rho\ge \frac\kappa 2-2$. A whole-plane SLE$(\kappa;\rho)$ process generates a simple trace, say $\gamma(t)$, $-\infty\le t<\infty$, which is  parameterized by whole-plane capacity. Recall the definition in Section \ref{Section-kappa-rho'}. There are $\lambda,q\in C(\R)$ such that $\lambda$ is the driving function, $q(t)$ satisfies the equation $q'(t)=\cot_2(q(t)-\lambda(t))$, and $Z(t):=\lambda(t)-q(t)\in(0,2\pi)$, $-\infty<t<\infty$, is a reversible stationary diffusion process with SDE: $dZ(t)=\sqrt\kappa dB(t)+(\frac\rho 2+1)\cot_2(Z(t))dt$.
Let $\mu_{\kappa;\rho}$ denote the invariant distribution for $(Z(t))$. Corollary \ref{Xtstationary'} shows that $\mu_{\kappa;\rho}$ has a density, which is proportional to $\sin_2(x)^{\frac 4\kappa(\frac\rho 2+1)}$. Corollary \ref{ergodic} shows that $(Z(t))$ is ergodic. Thus, for any $t_0\in\R$ and $f\in L^1(\mu_{\kappa;\rho})$, almost surely
\BGE \lim_{t\to-\infty} \frac{1}{t_0-t} \int_{t}^{t_0} f(Z(s))ds=\int f(x) d\mu_{\kappa;\rho}(x).\label{time-average}\EDE
We will proved that this property is preserved under conformal maps fixing $0$, as long as $f$ is uniformly continuous. The following lemma is obvious.

\begin{Lemma}
  Let $T_1,T_2\in\R$. Let $Z_j\in C((-\infty,T_j))$, $j=1,2$. Suppose that there is an increasing differential function $v$ defined on $(-\infty,T_1)$ such that $v((-\infty,T_1])=(-\infty,T_2]$, $v'(t)\to 1$ and $Z_2(v(t))-Z_1(t)\to 0$ as $t\to -\infty$. Let $f\in C(\R)$ be uniformly continuous. Then
  $$\lim_{t\to-\infty} \frac{1}{t_0-t} \int_{t}^{t_0} f(Z_1(s))ds=\lim_{t\to-\infty} \frac{1}{t_0-t} \int_{t}^{t_0} f(Z_2(s))ds $$
  as long as either limit exists and lies in $\R$ for some/every $t_0\in(-\infty, T_1\wedge T_2)$.  \label{Z-ergodic}
\end{Lemma}

We will need some properties of $\C$-hulls. Let $K$ be a $\C$-hull such that $\{0\}\subsetneqq K$. The following well-known fact follows from Schwarz lemma and Koebe's $1/4$ theorem (c.f.\ \cite{Ahl}):
\BGE e^{\ccap(K)}\le \max_{z\in K} |z|\le 4e^{\ccap(K)}.\label{diam-cap}\EDE

\begin{Lemma}
  For the above $K$, $|e^{\ccap(K)} g_K(z)-z|\le 5e^{\ccap(K)}$ for any $z\in\C\sem K$. \label{g_K-z}
\end{Lemma}
\begin{proof}
Since the derivative of $e^{\ccap(K)} g_K(z)$ at $\infty$ is $1$, $e^{\ccap(K)} g_K(z)-z$ extends analytically to $\ha\C\sem K$. Applying the maximum modulus principle, we see that $\sup_{z\in\C\sem K}|e^{\ccap(K)} g_K(z)-z|$ is approached by a sequence $(z_n)$ in $\C\sem K$ that tends to $K$. We have $|e^{\ccap(K)} g_K(z_n)|\to e^{\ccap(K)}$ and $\limsup |z_n|\le  \max_{z\in K}|z|$. The proof is completed by (\ref{diam-cap})
\end{proof}

Let $W$ be a conformal map, whose domain $\Omega$ contains $0$. Let $K$ be a $\C$-hull such that $\{0\}\subsetneqq K\subset\Omega$. Let $\Omega_K=g_K(\Omega\sem K)$, and define $W_K(z)=g_{W(K)}\circ W\circ g_K^{-1}(z)$ for $z\in\Omega_K$. Now $\Omega_K$ contains a neighborhood of $\TT$ in $\D^*$, and as $z\to\TT$ in $\Omega_K$, $W_K(z)\to\TT$ as well. Let $\Omega_K^\dagger=\Omega_K\cup\TT\cup I_\TT(\Omega_K)$. Schwarz reflection principle implies that $W_K$ extends to a conformal map on $\Omega_K^\dagger$ such that $W_K(\TT)=\TT$.


\begin{Lemma}
There are real constants $C_0<0$ and $C_1,C_2>0$ depending only on $\Omega$ and $W$ such that if $K$ is a $\C$-hull with $\{0\}\subsetneqq K$ and satisfies $\ccap(K)\le C_1$, then
\BGE |\ccap(W(K))-\ccap(K)-\log|W'(0)||\le C_1 e^{\frac 12\ccap(K)};\label{cap-W-K}\EDE
\BGE \log|W_K'(z)|\le C_2 e^{\frac 12\ccap(K)}/|\ccap(K)|,\quad z\in\TT.\label{WK'z}\EDE \label{W-K-est}
\end{Lemma}
\begin{proof}
Since $W(0)=0$ and $W'(0)\ne 0$, there is $V$ analytic in a neighborhood $\Omega'\subset\Omega$ of $0$ such that $V(0)=0$ and
$W(z)=W'(0)ze^{V(z)}$ in $\Omega'$.
There exist positive constants $C\ge 1$ and $\delta\le \frac 1{10}$ such that $|z|\le \delta$ implies that $z\in\Omega'$ and $|V(z)|\le C|z|$. Thus,
\BGE |W(z)|\ge |W'(0)||z|e^{-C|z|},\quad |W(z)-W'(0)z|\le |W'(0)||z|(e^{C|z|}-1),\quad |z|\le \delta.\label{W-z}\EDE

Suppose $K$ is a $\C$-hull with $\{0\}\subsetneqq K$, and satisfies $e^{\ccap(K)}\le  \delta^2\wedge \frac 1{(320C)^2}$.
From (\ref{diam-cap}) we see that $K\subset\{|z|\le 4\delta^2\}\subset\{|z|\le \delta\}\subset\Omega$. So $W(K)$ and $W_K$ are well defined.
Using (\ref{diam-cap}) and the connectedness of $K$, we may choose $z_0\in K$ such that $|z_0|=e^{\ccap(K)}$. Using (\ref{W-z}) we get $$|W(z_0)|\ge  |W'(0)||z_0|e^{-C|z_0|}\ge |W'(0)|e^{\ccap(K)}e^{-1/5}\ge \frac 45 |W'(0)|e^{\ccap(K)}.$$
 Since $W(z_0)\in W(K)$, using (\ref{diam-cap}) again, we get $\ccap(W(K))\ge \frac 14|W(z_0)|\ge \frac 15 |W'(0)|e^{\ccap(K)}$. Let $\alpha=\alpha_{W,K}=W'(0)e^{\ccap(K)-\ccap(W(K))}$. Then we have  $|\alpha|\le 5$.

 Let $R=\frac 12e^{-\frac 12\ccap(K)}$, $z_1\in\{|z|=R\}$, and $z_2=g_K^{-1}(z_1)$. From Lemma \ref{g_K-z}, we get
$$ |z_2-e^{\ccap(K)} z_1|\le 5 e^{\ccap(K)}. $$
Since $R\ge \frac 12(\delta^2)^{-1/2}\ge 5$, we have
$$ |z_2|\le (R+5)e^{\ccap(K)}\le 2Re^{\ccap(K)} =e^{\frac 12\ccap(K)}\le  \delta\wedge \frac1{360C}.$$
Let $J$ denote the Jordan curve $g_K^{-1}(\{|z|=R\})$, and $U_J$ denote its interior. Then $J\subset\{|z|\le \delta\}$, which implies that $U_J\subset\{|z|\le \delta\}\subset\Omega$. Since $g_K^{-1}$ maps the annulus $\{1<|z|\le R\}$ conformally onto $(J\cup U_J)\sem K\subset\Omega\sem K$, we see that $\{1<|z|\le R\}\subset \Omega_K$, and so $\{1/R\le |z|\le R\}\subset\Omega_K^\dagger$.
Let $z_3=W(z_2)$. Using (\ref{W-z}) and $0\le C|z_2|\le 1$, we get
$$ |z_3-W'(0)z_2|\le  |W'(0)||z_2|(e^{C|z_2|}-1)\le 2C |W'(0)||z_2|^2\le 2C|W'(0)| e^{\ccap(K)}.$$
Let $z_4=g_{W(K)}(z_3)$. From Lemma \ref{g_K-z} we get
$$ |z_4-e^{-\ccap(W(K))}z_3|\le 5.$$
Combining the above four displayed formulas and that $|\alpha|\le 5$, we get
 $$|z_4-\alpha z_1|\le 5+2 C |\alpha|+5|\alpha|\le 30+10C\le 40C.$$
 Note that $z_4=W_K(z_1)$. So we get
\BGE |W_K(z)-\alpha z|\le 40C,\quad |z|=R.\label{z4-z1}\EDE
\BGE |\alpha|R-40C\le|W_K(z)|\le |\alpha|R+40C,\quad |z|=R.\label{WKz}\EDE

We may find $R'>R$ such that $A:=\{1/R'<|z|<R'\}\subset \Omega_K^\dagger$. Then $W_K$ is analytic in $A$.
 Since $W_K$ is an orientation preserving auto homeomorphism of $\TT$, there is an analytic function $V_K$ such that $W_K(z)=e^{V_K(z)}z$ in $A$. We have $\Ree V_K(z)=\log|W_K(z)|-\log|z|$. Thus, $\Ree V_K\equiv 0$ on $\TT$.
 Cauchy's theorem implies that $\oint_{|z|=1} \frac{V_K(z)}zdz=\oint_{|z|=R} \frac{V_K(z)}zdz$, which means that $\int_0^{2\pi} V_K(e^{i\theta})d\theta=\int_0^{2\pi} V_K(Re^{i\theta})d\theta$. So we get
 $$0=\int_0^{2\pi} \Ree V_K(e^{i\theta})d\theta=\int_0^{2\pi} \Ree V_K(Re^{i\theta})d\theta=\int_0^{2\pi} (\log|W_K(Re^{i\theta})|-\log R)d\theta.$$
Using (\ref{WKz}), we get $ |\alpha|R-40C\le R\le |\alpha|R+40C$, which implies that $|1-|\alpha||\le \frac{40C}R$. This implies (\ref{cap-W-K}) since $\log|\alpha|=\log|W'(0)|+\ccap(K)-\ccap(W(K))$ and $1/R=O(e^{\frac 12\ccap(K)})$.

Let $|z|=R$. From (\ref{z4-z1}), we get $|e^{V_K(z)}-\alpha |\le \frac{40C}R$. Since $|\alpha |\ge 1-\frac{40C}R$, we have $|e^{V_K(z)}|\ge 1-\frac{80C}R\ge \frac 12$ as $R\ge 160C$. So there exists $\til \alpha\in\C$ with $\alpha=e^{\til \alpha}$ such that $|V_K(z)-\til \alpha|\le 2|e^{V_K(z)}-\alpha|\le \frac{80C}R$. From $||\alpha|-1|\le \frac{40C}R$, we get $|\Ree\til \alpha|=|\log|\alpha||\le \frac{80C}R$.
Thus, $|V_K(z)-i\Imm\til \alpha|\le \frac{160C}R$ if $|z|=R$.
Let $\til V_K=V_K\circ \exp$. Then $\til V_K$ is analytic in the vertical strip $\til A:=\exp^{-1}(A)=\{-\log R'<\Ree z<\log R'\}$, and is pure imaginary on $i\R$. Thus, $\til V_K(-\lin z)=-\lin{\til V_K(z)}$. This implies that, on the two vertical lines $\{\Ree z=\log R\}$ and $\{\Ree z=-\log R\}$, $|\til V_K(z)-i\Imm\til \alpha|\le  \frac{160C}R$. Since $\til V_K$ has period $2\pi i$, the inequality holds in the strip $\{-\log R\le \Ree z\le \log R\}$. We may apply Cauchy's integral formula, and get $|\til V_K'(z)|\le \frac{160C}{R\log R}$ for $z\in i\R$.
Since $\til V_K(z)=V_K\circ \exp$, $e^{V_K(z)}=\frac{W_K(z)}z$ and $W_K(\TT)=\TT$, we get$$
\Big |W_K'(z)-\frac{W_K(z)}z \Big|=|\til V_K'(\log z)|\le \frac{160C}{R\log R},\quad z\in\TT.$$
This implies (\ref{WK'z}) since $\log R\ge |\ccap(K)|/4$  and $1/R=O(e^{\frac 12\ccap(K)})$.
\end{proof}

Now suppose $\gamma(t)$, $-\infty\le t<T$, is a simple whole-plane Loewner trace driven by $\lambda\in C((-\infty,T))$. Let $\Omega$ be a domain that contains $\gamma$. Let $W$ be a conformal map defined on $\Omega$ such that $W(0)=0$. Let $\beta(t)=W(\gamma(t))$, $-\infty\le t<T$. Define $v$ on $[-\infty,T)$ such that $v(-\infty)=-\infty$ and $v(t)=\ccap(\beta([-\infty,t]))$ for $-\infty<t<T$. Let $\til T=v(T)$ and $\til\gamma(t)=\beta(v^{-1}(t))$, $-\infty\le t<\til T$. Then $\til\gamma$ is a simple whole-plane Loewner trace, say driven by $\til\lambda\in C((-\infty,\til T))$. Let $(g_t)$ and $(\til g_t)$ be the whole-plane Loewner maps driven by $\lambda$ and $\til\lambda$, respectively.
Then, $g_t^{-1}(e^{i\lambda(t)})=\gamma(t)$ and $\til g_t^{-1}(e^{i\til\lambda(t)})=\til \gamma(t)$.
Let $z(t)$ and $\til z(t)$ be  such that $g_t^{-1}(z(t))=0$ and $\til g_t^{-1}(\til z(t))=0$. Choose $q\in C((-\infty,T))$ and $\til q\in C((-\infty,\til T))$ such that $z(t)=e^{iq(t)}$, $\til z(t)=e^{i\til q(t)}$, $\lambda(t)-q(t)\in(0,2\pi)$, and $\til\lambda(t)-\til q(t)\in(0,2\pi)$.
Let $Z=\lambda-q$ and $\til Z=\til\lambda-\til q$.

Let $K_t=\gamma([-\infty,t])$ and $\til K_t=\til\gamma([-\infty,t])$. Recall that $g_t=g_{K_t}$ and $\til g_t=g_{\til K_t}$.
For $-\infty< t<T$, let $\Omega_t=\Omega_{K_t}$, $\Omega^\dagger_t=\Omega^\dagger_{K_t}$, and $W_t=W_{K_t}$. Then $W_t$ is a conformal map defined on $\Omega^\dagger_t\supset\TT$ such that $W_t(\TT)=\TT$. Since $W(K_t)=\til K_{v(t)}$, we have $W_t=\til g_{v(t)}\circ W\circ g_t^{-1}$ in $\Omega_t$. Since $g_t^{-1}(e^{i\lambda(t)})=\gamma(t)$ and $\til g_{v(t)}^{-1}(e^{i\til\lambda(v(t))})=\til \gamma(v(t))$ when both $g_t^{-1}$ and $\til g_{v(t)}$ extends continuously to $\D^*\cup\TT$, and $W(\gamma(t))=\til\gamma(v(t))$, we get $W_t(e^{i\lambda(t)})=e^{i\til\lambda(v(t))}$. Similarly, since  $g_t^{-1}(e^{iq(t)})=0=\til g_{v(t)}^{-1}(e^{i\til q(v(t))})$ and $W(0)=0$, we have $W_t(e^{i q(t)})=e^{i\til q(v(t))}$. Thus, we get
\BGE \til Z(v(t))= \til \lambda(v(t))-\til q(v(t))=\int_{q(t)}^{\lambda(t)} |W_t'(e^{is})|ds.\label{derivative-W}\EDE

\begin{Lemma}
  For any $t\in(-\infty,T)$, $v'(t)=|W_t'(e^{i\lambda(t)})|^2$. \label{v'}
\end{Lemma}
\begin{proof}
  Fix $t_0\in(-\infty,T)$. Let $t\in(0,T-t_0)$, and $K_{t_0+t,t_0}=I_\TT\circ g_{t_0}(\gamma((t_0,t_0+t]))$. Then $K_{t_0+t,t_0}$ is a $\D$-hull, and $g_{K_{t_0+t,t_0}}=I_\TT\circ g_{t_0+t}\circ g_{t_0}^{-1}\circ I_\TT$. Since $\lim_{z\to\infty} g_s(z)/z=e^{-s}$ for every $s$, we have $g_{K_{t_0+t,t_0}}'(0)=e^t$, which implies that $\dcap(K_{t_0+t,t_0})=t$. Let $\til K_{v(t_0+t),v(t_0)}=I_\TT\circ g_{v(t_0)}(\gamma((v(t_0),v(t_0+t)]))$. Then $\til K_{v(t_0+t),v(t_0)}$ is also a $\D$-hull, and $\dcap(\til K_{v(t_0+t),v(t_0)})=v(t_0+t)-v(t_0)$. Let $W^\TT_t=I_\TT\circ W_t\circ I_\TT$. Then $W^\TT_T$ is conformal in a neighborhood of $\TT$, maps $\TT$ onto $\TT$, and satisfies $W^\TT_t(K_{t_0+t,t_0})=\til K_{v(t_0+t),v(t_0)}$. Note that $(K_{t_0+t,t_0})$ is an increasing family in $t$, and satisfies $\bigcap_{t} \lin{K_{t_0+t,t_0}}=\{e^{i\lambda(t_0)}\}$. We now use the following well-known fact: for any $z_0\in\TT$, $\lim_{K\to z_0}\frac{\dcap (W^\TT_t(K))}{\dcap(K)}=|(W^\TT_t)'(z_0)|^2=|W_t'(z_0)|^2$, where $K\to z_0$ means that $K$ is a nonempty $\D$-hull, and $\diam(K\cup\{z_0\})\to 0$. This follows, e.g., from Lemma 2.1 in \cite{Zhan}. Thus, we have $\lim_{t\to 0^+} (v(t_0+t)-v(t_0))/t=|W_t'(e^{i\lambda(t_0)})|^2$. Finally, applying the maximum modulus principle, one can easily show that $(t,z)\mapsto W_t'(z)$ is continuous on $(-\infty,T)\times\TT$. So the proof is completed.
\end{proof}

Applying Lemma \ref{W-K-est} to $K=\gamma([-\infty,t])$ and using (\ref{derivative-W}) and Lemma \ref{v'}, we get
\BGE  \lim_{t\to-\infty} |\til Z(v(t))-Z(t)|=0,\quad \lim_{t\to-\infty} v'(t)=1,\quad \lim_{t\to-\infty} v(t)-t=\log|W'(0)|.\label{Z-Z}\EDE
Lemma \ref{Z-ergodic} implies that, if $f$ is continuous on $[0,2\pi]$, then
$$\lim_{t\to-\infty} \frac{1}{t_0-t} \int_t^{t_0} f(Z(s))ds=\lim_{t\to-\infty} \frac{1}{t_0-t} \int_t^{t_0} f(\til Z(s))ds,\quad t_0\in(-\infty, T\wedge \til T),$$ if either limit exists.
Using (\ref{time-average}) we obtain the following proposition.

\begin{Proposition}
Let $\kappa\le 4$ and $\rho\ge \frac\kappa 2-2$. Let $\gamma(t)$, $-\infty\le t<\infty$, be a whole-plane SLE$(\kappa;\rho)$ trace. Suppose that $W$ is a random conformal map with (random) domain $\Omega\ni 0$ such that $W(0)=0$. Let $T$ be such that $\gamma([-\infty,T))\subset \Omega$. Let $\til\gamma$ be a reparametrization of $W(\gamma(t))$, $-\infty\le t<T$, such that $\til\gamma(-\infty)=0$ and $\ccap(\til\gamma([-\infty,t]))=t$ for $-\infty<t<\til T$. Let $h(t)\in(0,1)$ denote the harmonic measure of the right side of $\til\gamma([-\infty,t])$ in $\ha\C\sem \til\gamma([-\infty,t])$ viewed from $\infty$. Then for any $f\in C([0,2\pi])$ and $t_0\in(-\infty,\til T)$, almost surely
$$\lim_{t\to -\infty} \frac{1}{t_0-t} \int_t^{t_0} f(2\pi h(s))ds=\int_0^{2\pi} f(x)d\mu_{\kappa;\rho}(x)
=\frac{\int_0^{2\pi}f(x)\sin_2(x)^{\frac 4\kappa(\frac\rho 2+1)}dx }{\int_0^{2\pi}\sin_2(x)^{\frac 4\kappa(\frac\rho 2+1)}dx }.$$ \label{prop3}
\end{Proposition}

Combining the above proposition with Theorems \ref{Thm2} and \ref{Thm-R}, we obtain the following theorem.

\begin{Theorem}
  Let $\kappa\in(0,4)$ and $t_0\in(0,\infty)$. Let $\beta$ be a chordal or radial SLE$_\kappa$ trace. For $0\le t< t_0$, let $v(t)=\ccap(\beta([t,t_0]))$ and $h(t)$ be the harmonic measure of the left side of $\beta([t,t_0])$ in $\ha\C\sem \beta([t,t_0])$ viewed from $\infty$. Then for any $f\in C([0,2\pi])$, almost surely
  $$\lim_{t\to t_0^-} \frac{1}{v(t)-v(0)} \int_0^t f(h(s))dv(s) =\frac{\int_0^{2\pi}f(x)\sin_2(x)^{\frac8\kappa+2}dx }{\int_0^{2\pi}\sin_2(x)^{\frac 8\kappa+2}dx }.$$ \label{chordal-radial-ergodicity}
\end{Theorem}

\no{\bf Remarks.} 
\begin{enumerate}
  \item We can now conclude that Theorem \ref{reversibility-whole} does not hold with $\kappa+2$ replaced by any other $\rho\ge \frac\kappa 2-2$. If this is not true, then Theorem \ref{Thm-R} also holds with $\kappa+2$ replaced by such $\rho$. Then Theorem \ref{chordal-radial-ergodicity} holds in the radial case with the exponent $\frac8\kappa+2$ replaced by $\frac 4\kappa(\frac{\rho} 2+1)$, which is obviously impossible.
\item Fubini's Theorem implies that Theorem \ref{chordal-radial-ergodicity} still holds if the deterministic number $t_0$ is replaced a positive random number $\lin t_0$, whose distribution given $\beta$ is absolutely continuous with respect to the Lebesgue measure. 
    We do not expect that the theorem holds if the conditional distribution of $\lin t_0$ does not have a density. In fact, if the conditional distribution of $\lin t_0$ is absolutely continuous with respect to the natural parametrization introduced by Lawler and Sheffield \cite{NP}, then we expect that $\beta$ behaves like a two-sided radial SLE$_\kappa$ process, which is a radial SLE$(\kappa;2)$ process, near $\beta(\lin t_0)$, and  Theorem \ref{chordal-radial-ergodicity} is expected to hold with $\frac 8\kappa+2$ replaced by $\frac 8\kappa$.
\end{enumerate}

Let $\kappa\in(0,4]$. A whole-plane SLE$(\kappa;\rho)$ trace $\gamma$ generates a simple curve. Combining the reversibility property derived in \cite{whole} with the Markov-type relation between whole-plane SLE$_\kappa$ and radial SLE$_\kappa$ processes, we see that, if $\beta$ is a radial SLE$_\kappa$, there is a conformal map $V$ defined on $\D$ with $V(0)=0$, which maps $\beta$ to an initial segment of a whole-plane SLE$_\kappa$ trace. Applying Proposition \ref{prop3}, we obtain the following.

\begin{Theorem}
  Let $\kappa\in(0,4]$. Let $\beta$ be a  radial SLE$_\kappa$ trace. For $0\le t<\infty$, let $v(t)=\ccap(\beta([t,\infty]))$ and $h(t)$ be the harmonic measure of the left side of $\beta([t,\infty])$ in $\ha\C\sem \beta([t,\infty])$ viewed from $\infty$. Then for any $f\in C([0,2\pi])$, almost surely
  $$\lim_{t\to \infty} \frac{1}{v(t)-v(0)} \int_0^t f(h(s))dv(s) =\frac{\int_0^{2\pi}f(x)\sin_2(x)^{\frac 4\kappa}dx }{\int_0^{2\pi}\sin_2(x)^{\frac 4\kappa}dx }.$$
\end{Theorem}

\appendixpage
\addappheadtotoc
\appendix
\section{Carath\'eodory Convergence} \label{A}
\begin{Definition}  Let $(D_n)_{n=1}^\infty$ and $D$ be domains in a Rieman surface $R$. We say that $(D_n)$ converges to $D$ in the Carath\'eodory topology, and write $D_n\dto D$, if 
\begin{enumerate}
  \item [(i)] for every compact set $K\subset D$, there exists $n_0\in\N$ such that $K\subset D_n$ if $n\ge n_0$;
\item [(ii)] for every point $z_0\in\pa D$, there exists $z_n\in\pa D_n$ for each $n$ such that $z_n\to z_0$.
\end{enumerate} \label{def-lim}
\end{Definition}

\no {\bf Remark.}
 A sequence of domains may converge to two different domains. For example, let $D_n=\C\sem((-\infty,n])$. Then $D_n\dto\HH$, and
$D_n\dto -\HH$ as well. But two different limit domains of the same domain sequence must be disjoint from each other, because if they
have nonempty intersection, then one contains some boundary point of the other, which implies a contradiction.


\begin{Lemma} Let $R$ and $S$ be two Riemann surfaces. Let $D_n$, $n\in\N$, and $D$ be domains in $R$ such that $D_n\dto D$. Let $f_n$ map $D_n$ conformally into $S$, $n\in\N$. Suppose $(f_n)$ converges locally uniformly in $D$.  Assume that the limit function $f$ is not constant in $D$. Then $f$ is a conformal map, $f(D_n)\dto f(D)$, and $f_n^{-1}\luto f^{-1}$ in $f(D)$. \label{domain convergence*}
\end{Lemma}

\no{\bf Remark.} The lemma improves the Carath\'eodory kernel theorem (Theorem 1.8, \cite{Pom-bond}) such that the domains do not have to be simply connected. A simpler version (in the case $R$ and $S$ are $\C$ or $\ha\C$) was introduced in \cite{LERW}, and used in the author's other papers, but no proof has been given so far. For completeness, we include the proof here.

\begin{proof}
Cauchy-Goursat theorem implies that $f$ is analytic. We first prove that $f$ is one-to-one. Assume that $f$ is not one-to-one. Then there exist $z_1\ne z_2\in D$ such that $f(z_1)=f(z_2):=w_0$. Since $f$ is not constant, $f^{-1}(w_0)$ has no accumulation points in the domain $D$.
Let $(V,\psi)$ be a chart for $S$ such that $w_0\in V$ and $\psi(w_0)=0$. We may find charts $(U_1,\phi_1)$ and $(U_2,\phi_2)$ for $R$ such that $z_j\in U_j\subset D$, $f(U_j)\subset V$, $\phi_j(z_j)=0$, $\phi_j(U_j)\supset \lin\D$, $\phi_j^{-1}(\TT)\cap f^{-1}(w_0)=\emptyset$, $j=1,2$, and $U_1\cap U_2=\emptyset$. Since $D_n\dto D$, we have $\phi_j^{-1}(\lin\D)\subset D_n$, $j=1,2$, if $n$ is big enough. Thus, for $j=1,2$, $\psi\circ f_n\circ \phi_j^{-1}$ tends uniformly on $\lin\D$ to $\psi\circ f\circ \phi_j^{-1}$, which has a zero at $0$ and has no zero on $\TT$. Rouch\'e's theorem implies that when $n$ is big enough, $\psi\circ f_n\circ \phi_j^{-1}$ has zero(s) in $\D$ for $j=1,2$, which implies that $f_n^{-1}(w_0)$ intersects both $U_1$ and $U_2$. This contradicts that each $f_n$ is one-to-one, and $U_1\cap U_2=\emptyset$. So $f$ is one-to-one.

Let $E_n=f(D_n)$, $n\in\N$, and $E=f(D)$ be domains in $S$. Since $f_n\luto f$ in $D$, we have $f_n\circ f^{-1}\luto \id$ in $E$. Let $K\subset E$ be a closed ball, which means that there is a chart $(V,\psi)$ for $S$ such that $K\subset V\subset E$ and $\psi(K)=\{|z|\le r_0\}$ for some $r_0>0$. We may choose $r_1>r_0$ such that $\psi(V)\supset \{|z|\le r_1\}$. Let $K'=\psi^{-1}(\{|z|\le r_1\})$. Applying Rouch\'e's theorem to the Jordan curve $\{|z|=r_1\}$ and the functions $\psi\circ f_n\circ f^{-1}\circ \psi^{-1}(z)-z_0$ and $z-z_0$, where $z_0\in\{|z|\le r_0\}$, we see that when $n$ is big enough, $\psi\circ f_n\circ f^{-1}\circ \psi^{-1}(z)-z_0$ has a zero in $\{|z|<z_1\}$ for every $z_0\in\{|z|\le r_0\}$, which implies that $K=\psi^{-1}(\{|z|\le r_0\})\subset f_n(D_n)=E_n$. Since every compact subset of $E$ can be covered by finitely many closed balls in $E$, condition (i) in Definition \ref{def-lim} holds for $E_n$ and $E$.

Let $g_n=f_n^{-1}$, $n\in\N$, and $g=f^{-1}$.
Now we prove that $g_n\luto g$ in $E$. Assume that this is not true. By passing to a subsequence, we may find a sequence $(w_n)$ in $E$ with $w_n\to w_0\in E$ such that $g(w_0)$ is not any subsequential limit of $(g_n(w_n))$. Let $(V,\psi)$ be a chart for $S$ such that $w_0\in V\subset E$ and $\psi(w_0)=0$. Let $r_1>0$ be such that $\{|z|\le r_1\}\subset \psi(V)$; and let $r_0\in(0,r_1)$. Since $w_n\to w_0$, there is $n_0\in\N$ such that $\psi(w_{n})\in\{|z|\le r_0\}$ for $n\ge n_0$. The argument in the previous paragraph shows that, there is $n_1\in\N$ such that, if $n\ge n_1$, then for every $z\in\{|z|\le r_0\}$, there is $z'\in\{|z|<r_1\}$ such that $\psi\circ f_n\circ g\circ \psi^{-1}(z')=z$. Taking $z=\psi(w_n)$, we see that $g_n(w_n)\in g\circ \psi^{-1} (\{|z|< r_1\})$  for $n\ge n_0\vee n_1$. Since $r_1>0$ can be chosen arbitrarily small and $\psi^{-1}(0)=w_0$, this contradicts that $g(w_0)$ is not any subsequential limit of $(g_n(w_n))$. Thus, $g_n\luto g$ in $E$.

It remains to prove that condition (ii) in Definition \ref{def-lim} holds for $E_n$ and $E$. Assume that this is not true. By passing to a subsequence, we may assume that there exist $w_0\in\pa E$ and a domain $V$ with $w_0\in V\subset S$ such that $V\cap\pa E_n=\emptyset$ for each $n$. Let $w_0'\in E\cap V$. Since condition (ii) in Definition \ref{def-lim} holds for $E_n$ and $E$, if $n$ is big enough, then $w_0'\in E_n$, which implies that $V\subset E_n$ because $V\cap\pa E_n=\emptyset$ and $V$ is connected. By removing finitely many terms, we may assume that $V\subset E_n$ for each $n$. By considering a smaller $V$, we may further assume that there is $\psi:V\conf 2\D$ such that $\psi(w_0)=0$. We will restrict our attention to $V$ and derive a contradiction. So we may assume that $V=2\D$, $\psi=\id$, and $w_0=0$.

It is well known that there is an increasing function $h(r)$ defined on $(0,1)$ with $h(0^+)=0$ such that the probability that a planar Brownian motion started from $0$ hits $\TT$ before disconnecting $r\TT$ from $\TT$ is less than $h(r)$. Pick $r_0\in(0,1/5)$ such that $h(r_0)+h(5r_0)<1$.

Since $w_0=0\in\pa E$, may find $w_1\in E\cap V$ such that $|w_1|<0.1\wedge r_0$. Let $s\in (0,0.1)$ be such that $U_2:=\{|w-w_1|<s\}\subset E$. Let $U_1=\{|w-w_1|< s/2\}$. Since $g_n\luto g$ in $U_2$, from what we have derived, condition (i) in Definition \ref{def-lim} holds for $g_n(U_2)$ and $g(U_2)$. Thus, there is $n_0\in\N$ such that $g_n(w_1)\in g(U_1)\subset g(\lin{U_1})\subset g_n(U_2)$ when $n\ge n_0$. This implies that, if $n,m\ge n_0$, then $f_n\circ g_m(w_1)\in U_2$, i.e., $|f_n\circ g_m(w_1)-w_1|<s<0.1$, and so $|f_n\circ g_m(w_1)|<0.2$.

That $g_n\luto g$ in $E$ also implies that $g_n'(w_1)\to g'(w_1)\in\C\sem\{0\}$. So there is $n_1\ge n_0$ such that, if $n,m\ge n_1$ then $|(f_n\circ g_m)'(w_1)|\in(0.9,1.1)$. Fix $n,m\ge n_1$. Let $W=f_n\circ g_m$ and $w_2=W(w_1)$. Recall that $|w_1|<0.1$ and $|w_2|<0.2$. So $w_1+\D$ and $w_2+\D$ are contained in $2\D=V\subset E_n\cap E_m$. Let $\Omega_1=f_m(g_m(w_1+\D)\cap g_n(w_2+\D))\subset w_1+\D$ and $\Omega_2=f_n(g_m(w_1+\D)\cap g_n(w_2+\D))\subset w_2+\D$. Then $w_j\in\Omega_j$, $j=1,2$, and $W:\Omega_1\conf \Omega_2$.

Let $r_j=\dist(w_j,\pa\Omega_j)$. Since $|W'(w_1)|\in (0.9,1.1)$, Koebe's $1/4$ theorem implies that $r_2<4.4 r_1$. Let $I_1= W^{-1}(w_2+\TT)\cap(w_1+\D)$ and $I_2=(w_1+\TT)\cap W^{-1}(w_2+\D)$. Then $I_1$ and $I_2$ are disjoint subsets of $\pa\Omega_1$. For $k=1,2$, let $h_k$ be the harmonic measure of $I_k$ in $\Omega_1$ viewed from $w_1$. Then $h_1+h_2\le 1$. Note that $\pa\Omega_1\sem I_1\subset\TT$, and $I_1$ contains a connected component, which touches both $w_1+\TT$ and $w_1+r_1\TT$. So $h_1\ge 1-h(r_1)$. Let $I_2'=W(I_2)=W(w_1+\TT)\cap (w_2+\D)\subset\pa \Omega_2$. Then $\pa\Omega_2\sem I_2'\subset\TT$, and $I_2'$ contains a connected component, which touches both $w_2+\TT$ and $w_2+r_2\TT$. From conformal invariance of harmonic measures, $h_2$ is equal to the harmonic measure of $I_2'$ in $\Omega_2$ viewed from $w_2$, which is at least $1-h(r_2)$. Thus, we have $1\ge h_1+h_2\ge (1-h(r_1))+(1-h(r_2))$, from which follows that $1\le h(r_1)+h(r_2)$. If $r_1<r_0$, since $h$ is increasing and $r_2<4.4 r_1$, we get $h(r_1)+h(r_2)\le h(r_0)+h(5r_0)<1$, which is a contradiction. So $r_1\ge r_0$.

So we conclude that, for any $n,m\ge n_1$, $f_m\circ g_n$ is well defined and analytic on $U_0:=\{|w-w_1|<r_0\}$. Fix $m=n_1$. Since $f_{n_1}\circ g_{n}(w_1)\to f_{n_1}\circ g(w_1)$ and $(f_{n_1}\circ g_{n})'(w_1)\to (f_{n_1}\circ g)'(w_1)$, Koebe's distortion theorem implies that $(f_{n_1}\circ g_{n}|_{U_0})_{n\ge n_1}$ is a normal family. Since $f_{n_1}\circ g_{n}\luto f_{n_1}\circ g$ in $E \cap U_0$, we see that $f_{n_1}\circ g_{n}$ converges locally uniformly in $U_0$, as $n\to\infty$, and the limit is an analytic extension of $f_{n_1}\circ g$ from $E\cap U_0$ to $U_0$. Thus, $g$ extends analytically to $E':=E\cup U_0$, and $g_n\luto  g$ in $E'$. Since $|w_1|<r_0$, we have $w_0=0\in U_0\cap\pa E$. Thus, $z_0:=g(w_0)\in\pa D$. Let $K$ be a compact subset of $U_0$, whose interior $\mathring K$ contains $w_0$. Since $g_n\luto g$ in $U_0$, from what we have derived, condition (i) in Definition \ref{def-lim} holds for $g_n(U_0)$ and $g(U_0)$. Thus, $z_0\in g(\mathring K)\subset g(K)\subset g_n(U_0)\subset D_n $ when $n$ is big enough, which contradicts that $z_0\in\pa D_n$ and $D_n\dto D$ as $g(\mathring K)$ is an open set. The contradiction completes the proof.
\end{proof}

\no{\bf Remark.} The only place that we use the connectedness is that $f$ is not constant implies $f^{-1}(w_0)$ has no accumulation points. Thus, we may define Carath\'eodory convergence of open sets in a Riemann surface. Lemma \ref{domain convergence*} still holds when $D_n$ and $D$ are not domains, if the condition that $f$ is not constant is replaced by that $f$ is not locally constant.


\section{Radial Bessel Processes} \label{B'}
Let $\delta\in\R$. Consider the SDE:
\BGE dX_t=dB_t+\frac{\delta-1}2\cot(X_t)dt,\quad X_0\in (0,\pi).\label{dXt'}\EDE
The solution is called a radial Bessel process of dimension $\delta$. The name comes from the fact that the process arises in the definition of radial SLE$(\kappa;\rho)$ processes, and $(X_t)$ behaves like a Bessel process of dimension $\delta$ when it is close to $0$ or $\pi$. Let $[0,T)$ denote the time interval for $(X_t)$. Define $h(x)=\int_{\pi/2}^x \sin(t)^{1-\delta}dt$, $0<x<\pi$. It\^o's formula (c.f.\ \cite{RY}) shows that $h(X_t)$, $0\le t<T$, is a local martingale. Note that $h((-1,1))=\R$ if $\delta\ge 2$; and is bounded if $\delta<2$. A simple argument shows that, if $\delta\ge 2$, then $T=\infty$; if $\delta< 2$, then $T<\infty$ and $\lim_{t\to T}X_t\in\{0,\pi\}$.
Let $Y_t=\cos(X_t)$, $0\le t<T$. It\^o's formula shows that
\BGE dY_t=-\sqrt{1-Y_t^2}dB(t)-\frac\delta2 Y_t dt, \quad 0\le t<T.\label{dYt2'}\EDE

Suppose $\delta\ge 2$. We will derive the transition densities of $(Y_t)$ and $(X_t)$.  Observe that if the process $(Y_t)$ has a smooth transition density $p(t,x,y)$, then it satisfies the Kolmogorov's backward equation:
\BGE \pa_t p =\frac{1-x^2}2 \pa_x^2 p  -\frac{\delta }2 x \pa_x p.\label{PDE'}\EDE
Below we will solve (\ref{PDE'}) using the eigenvalue method, and prove that some solution is the transition density of $(Y_t)$.

Let $\lambda\in\R$. Consider the ODE:
\BGE ({1-x^2}) p''(x)-{\delta x} p'(x)-2\lambda p(x)=0.\label{ODE'}\EDE
If $\lambda=\lambda_n=-\frac n2(n+\delta-1)$, $n\in\N\cup\{0\}$, the above equation has a solution, which is the Gegenbauer polynomial $C_n^{(\alpha)}(x)$ (c.f.\ \cite{orthogonal}) with degree $n$ and index $\alpha:=\frac{\delta}2-\frac 12$. Thus,
$p_n(t,x):=e^{-\frac n2(n+\delta-1)t} C_n^{(\frac{\delta}2-\frac 12)}(x)$, $n\in\N\cup\{0\}$, solve (\ref{PDE'}) for $t,x\in\R$.

The functions $C_n^{(\alpha)}(x)$, $n\in\N\cup\{0\}$ form a complete orthogonal system w.r.t.\ the inner product
$\langle f,g\rangle_{\alpha-\frac 12}:=\int_{-1}^1 (1-x^2)^{\alpha-\frac 12} f(x)g(x) dx$
such that $\langle C_n^{(\alpha)},C_m^{(\alpha)}\rangle_{\alpha-\frac 12}=0$ when $n\ne m$, and
\BGE \langle C_n^{(\alpha)},C_n^{(\alpha)}\rangle_{\alpha-\frac 12}=\frac{\pi \Gamma(2\alpha+n)}{2^{2\alpha-1}(\alpha+n)n!\Gamma(\alpha)^2} \sim n^{2\alpha-2}.\label{norm-n'}\EDE
Moreover,
\BGE \|C_n^{(\alpha)}\|_\infty:=\max_{-1\le x\le 1} |C_n^{(\alpha)}(x)|=\frac{\Gamma(n+2\alpha)}{n!\Gamma(2\alpha)}\sim   n^{2\alpha-1}.\label{supernorm0'}\EDE
For $t>0$, $x,y\in[-1,1]$, define
\BGE p^{(Y)}(t,x,y)= \sum_{n=0}^\infty \frac{(1-y^2)^{\frac\delta 2-1} C_n^{(\frac{\delta}2-\frac 12)}(x)C_n^{(\frac{\delta}2-\frac 12)}(y)}{\int_{-1}^1 (1-y^2)^{\frac{\delta}2-1} C_n^{(\frac{\delta}2-\frac 12)}(y)^2dy} \exp({-\frac n2(n+\delta-1)t}).\label{pY'}\EDE
From (\ref{norm-n'}) and (\ref{supernorm0'}) we see that the above series converges uniformly on $[-1,1]$.

\begin{Proposition}
If $\delta\ge 2$, the transition density for $(Y_t)$ is $p^{(Y)}(t,x,y)$ given by (\ref{pY'}), and the transition density for $(X_t)$ is $p^{(X)}(t,x,y)=p^{(Y)}(t,\cos x,\cos y)\sin y$.  \label{densityY'}
\end{Proposition}
\begin{proof}
It suffices to derive the the transition density for $(Y_t)$.
Let $f(x)$ be a polynomial, and $a_n=\langle f, C_n^{(\frac{\delta}2-\frac 12)}\rangle_{\frac\delta2-1}/\langle C_n^{(\frac{\delta}2-\frac 12)}, C_n^{(\frac{\delta}2-\frac 12)}\rangle_{\frac\delta2-1}$, $n\in\N\cup\{0\}$. Then all but finitely many $a_n$'s are zero, and $f=\sum_{n=0}^\infty a_n C_n^{(\frac{\delta}2-\frac 12)}$. Define $ f(t,x)=\sum_{n=0}^\infty a_n p_n(t,x)$.
Then $f(t,x)$ solves (\ref{PDE'}) with $f(0,x)=f(x)$. Suppose $(Y_t)$ solves (\ref{dYt2'}) with initial value $x_0$.
Fix $t_0>0$. It\^o's formula together with the boundedness of $f(t,x)$ on $[0,t_0]\times[-1,1]$ shows that $M(t):=f(t_0-t,Y_t)$, $0\le t<t_0$, is a bounded martingale. Since $\lim_{t\to t_0} M(t)=f(Y_{t_0})$, the optional stopping theorem together with the definition of $p^{(Y)}(t,x,y)$ implies that
$$ \EE_{x_0}[f(Y_{t_0})]=M(0)=f(t_0,x_0)=\int_{-1}^1 f(y) p^{(Y)}(t_0,x_0,y)dy.$$ 
Since this holds for any polynomial $f$, the proof is finished.
\end{proof}

\begin{Corollary}
  Let $\delta\ge 2$. Then $(Y_t)$ has a unique stationary distribution, which has a density
  \BGE  p^{(Y)}(x)=\frac{(1-x^2)^{\frac\delta 2-1}}{\int_{-1}^1 (1-y^2)^{\frac\delta 2-1}dy},\quad x\in(-1,1);\EDE
  and $(X_t)$ has a unique stationary distribution, which has a density $p^{(X)}(x)=p^{(Y)}(\cos x)\sin x$, $x\in(-\pi,\pi)$. Moreover, the stationary processes $(Y_t)$ and $(X_t)$ are reversible. \label{Xtstationary'}
\end{Corollary}
\begin{proof} This follows from the previous proposition and the orthogonality of $C_n^{(\frac{\delta}2-\frac 12)}$ w.r.t.\ $\langle\cdot\rangle_{\frac\delta 2-1}$. Note that $C_0^{(\frac{\delta}2-\frac 12)}\equiv 1$ and $(1-x^2)^{\frac\delta 2-1}p^{(Y)}(t,x,y)=(1-y^2)^{\frac\delta 2-1}p^{(Y)}(t,y,x)$.
\end{proof}


Note that $p^{(Y)}(y)$ is also  the term for $n=0$ in (\ref{pY'}). Using (\ref{norm-n'}) and (\ref{supernorm0'}), we see that there is a constant $C$ depending on $\delta$ such that
  \BGE | p^{(Y)}(t,x,y)- p^{(Y)}(y)|\le C e^{-\frac \delta 2 t},\quad x,y\in[-1,1].\EDE
  Thus, $p^{(Y)}(t,x,y)\to p^{(Y)}(y)$ as $t\to\infty$ uniformly in $x,y\in[-1,1]$. So we obtain the following corollary.
\begin{Corollary}
  Let $\delta\ge 2$. Then the stationary processes $(Y_t)$ and $(X_t)$ are mixing, and so are ergodic. \label{ergodic}
\end{Corollary}

\vskip 4mm
We now study the transition densities in the case $\delta<2$. Recall that $[0,T)$ is the time interval for $(Y_t)$.
We say that $\til p^{(Y)}(t,x,y)$ is the transition density of $(Y_t)$ if for any $f\in C([-1,1])$,
\BGE\EE_x[{\bf 1}_{T>t} f(Y_t)]=\int_{-1}^1 f(y) \til p^{(Y)}(t,x,y)dy,\quad x,y\in(-1,1),t>0.\label{killed'}\EDE
The integral $\int_{-1}^1 \til p(t,x,y)dy=\EE_x[T>t]$ may be less than $1$.

We will need functions, which solve (\ref{PDE'}) for $x\in(-1,1)$ and vanish at $x\in\{-1,1\}$. It is easy to see that if $p(x)=(1-x^2)^{1-\frac\delta 2}q(x)$, then $p(x)$ solves (\ref{ODE'}) in $(-1,1)$ iff $q(x)$ solves
$$ ({1-x^2}) q''(x)-(4- {\delta }) xq'(x)-(2\lambda+2-\delta ) q(x)=0,\quad -1<x<1.$$
If $\lambda=-\frac 12(n+1)(n+2-\delta)$, $n\in\N\cup\{0\}$, the above equation has a solution $C_n^{(\frac 32-\frac\delta2)}$.
Thus, $$\til p_n(t,x):=(1-x^2)^{1-\frac\delta 2}C_n^{(\frac 32-\frac\delta2)} e^{-\frac 12(n+1)(n+2-\delta)t}$$
solves (\ref{PDE'}) for $x\in(-1,1)$ and vanishes at $x\in\{-1,1\}$.

Note that $C_n^{(\frac 32-\frac\delta2)}$, $n\in\N\cup\{0\}$, form a complete orthogonal system w.r.t.\ $\langle \cdot\rangle_{1-\frac\delta 2}$. So we define
\BGE \til p^{(Y)}(t,x,y)=\sum_{n=0}^\infty  \frac{(1-x^2)^{1-\frac\delta 2} C_n^{(\frac 32-\frac\delta2)}(x) C_n^{(\frac 32-\frac\delta2)}(y)}{\int_{-1}^1 (1-y^2)^{1-\frac\delta 2} C_n^{(\frac 32-\frac\delta2)}(y)^2dy }\exp(-\frac 12(n+1)(n+2-\delta)t).\label{tilp'}\EDE
Let $P$ be a polynomial, and $a_n=\langle P, C_n^{(\frac 32-\frac\delta2)}\rangle_{1-\frac\delta2}/\langle C_n^{(\frac 32-\frac\delta2)}, C_n^{(\frac 32-\frac\delta2)}\rangle_{1-\frac\delta2}$, $n\in\N\cup\{0\}$. Then all but finitely many $a_n$'s are zero, and $P=\sum_{n=0}^\infty a_n C_n^{(\frac 32-\frac\delta2)}$. Define $\til f(t,x)=\sum_{n=0}^\infty a_n \til p_n(t,x)$. Then $\til f(t,x)$ solves (\ref{PDE'}) for $x\in(-1,1)$, vanishes at $x\in\{-1,1\}$, and satisfies $\til f(0,x)=f(x):=(1-x^2)^{1-\frac\delta 2}P(x)$. Fix $t_0>0$. Define $\til M_{t}:=\til f(t_0-t,Y_t)$, $0\le t\le T$. Then $\til M_t$ is a martingale with $\til M_T=0$. The optional stoping theorem implies that
$$ \EE_{x_0}[{\bf 1}_{T>t_0}\til f(Y_{t_0})]=\EE_{x_0}[M_{T\wedge t_0}]=M_0=\til f(t_0,x_0)=\int_{-1}^1 f(y)\til p^{(Y)}(t_0,x_0,y)dy.$$
Thus (\ref{killed'}) holds for $f(x)=(1-x^2)^{1-\frac\delta 2}P(x)$. Then a denseness argument show that (\ref{killed'}) holds for any $f\in C([-1,1])$. So we obtain the following proposition.

\begin{Proposition}
   Let $\delta<2$. The transition density of $(Y_t)$ is $\til p^{(Y)}(t,x,y)$ given by (\ref{tilp'}), and the transition density of $(X_t)$ is $\til p^{(Y)}(t,\cos x,\cos y)\sin y$. 
\end{Proposition}

Note that the term for $n=0$ in (\ref{tilp'}) is
\BGE \til p^{(Y)}(t,x):= \frac{(1-x^2)^{1-\frac\delta 2}}{\int_{-1}^1 (1-y^2)^{1-\frac\delta 2}dy}\, e^{-\frac 12(2-\delta)t}.\EDE
 Using (\ref{norm-n'}) and (\ref{supernorm0'}), we see that there is a constant $C$ depending on $\delta$ such that
  \BGE |\til p^{(Y)}(t,x,y)-\til p^{(Y)}(t,x)|\le C  e^{- (3-\delta)t},\quad x,y\in(-1,1).\EDE
Since $\PP_x^{(Y)}[T>t]=\int_{-1}^1 \til p^{(Y)}(t,x,y)dy$, using the fact that $C_1^{(\alpha)}(y)=2\alpha y$ is odd we see that there is a constant $C$ depending on $\delta$ such that
 \BGE |\PP_x^{(Y)}[T>t]-2\til p^{(Y)}(t,x)|\le C  e^{-\frac 32(4-\delta)t},\quad x\in(-1,1).\EDE
So we obtain the following corollary. 

\begin{Corollary}
  Let $\delta<2$, and $T$ be the lifetime for $(Y_t)$ or $(X_t)$. Then for any initial values, $\PP^{(Y)}[T>t]$ and $\PP^{(X)}[T>t]$ are bounded above by a constant depending on $\delta$ times $e^{-\frac 12(2-\delta)t}$, and for any $a<\frac 12(2-\delta)$, $\EE^{(Y)}[e^{aT}]$ and $\EE^{(X)}[e^{aT}]$ are finite.
\end{Corollary}

\no{\bf Remarks.}
\begin{enumerate}
\item Gregory Lawler has a method to prove Corollary \ref{Xtstationary'} without finding the transition density (Appendix A, \cite{Law-Bessel}). The idea is to use Girsanov's theorem to compare a radial Bessel process of dimension $\delta\ge 2$ with a Brownian motion. His method also works for some functions other than $\frac{\delta-1}2\cot(x)$.
\item We may define a radial Bessel process $(X_t)$ with dimension $\delta\in[0,2)$ such that the time interval is $[0,\infty)$. First, we define $(Y_t)$ to be the solution of the SDE: $dY_t=-q(Y_t)dB(t)-\frac\delta2 Y_t dt$ with $Y_0\in(-1,1)$, where $q(x)=\sqrt{(1-x^2)\vee 0}$. Since $q$ is H\"older $1/2$ continuous, the existence and uniqueness of the strong solution defined on $[0,\infty)$ follow from Theorems 1.7 and 3.5 in $\S$IX of \cite{RY}. If $\delta\ge 0$, then $(Y_t)$ stays on $[-1,1]$, and so solves (\ref{dYt2'}). Then the process $(X_t)$ is defined by $X_t=\arccos(Y_t)$. Proposition (\ref{densityY'}) and its two corollaries also hold for $\delta\in(0,2)$ because the functions $p_n(t,x,y)$ solve (\ref{PDE'}) for all $x\in\R$. Lawler's argument does not work in this case since Girsanov's theorem does not apply.
\item We may also consider the transition density of the process $(Y_t)$, which solves the SDE
  $$dY_t=-\sqrt{1-Y_t^2}dB(t)-\frac{\delta_+}4(Y_t+1)dt-\frac{\delta_-}4(Y_t-1)dt,\quad Y_0\in(-1,1).$$
If $\delta_+=\delta_-=\delta$, this SDE becomes (\ref{dYt2'}). 
If $\delta_+,\delta_->0$, then $(Y_t)$ stays in $[-1,1]$, and the transition density is given by (\ref{pY'}) revised such that $C_n^{(\frac{\delta} 2-\frac 12)}$ is replaced by the Jacobi polynomial $P_n^{(\frac{\delta_+}2-1,\frac{\delta_-}2-1)}$, the weight $(1-y^2)^{\frac\delta 2-1}$ is replaced by $(1-y)^{\frac{\delta_+}2-1}(1+y)^{\frac{\delta_-}2-1}$, and the number $n+\delta-1$ is replaced by $n+\frac{\delta_++\delta_-}2-1$. Such $(Y_t)$ has a unique stationary distribution with density proportional to $(1-x)^{\frac{\delta_+}2-1}(1+x)^{\frac{\delta_-}2-1}$, and the corresponding stationary process is reversible, mixing and ergodic. 
One may also use the Jacobi polynomials to express the transition density of the process $(Y_t)$ killed after it hits $\{-1,1\}$ in the case $\delta_+$ or $\delta_-$ is less than $2$, which resembles (\ref{tilp'}). Such process $(Y_t)$ was studied in Section 4 of \cite{SS}.
\end{enumerate}

\end{document}